\documentclass{amsart}
\usepackage[disable]{todonotes}
\usepackage{todonotes}
\usepackage{hyperref}
\usepackage{amssymb,amsthm,mathabx}
\usepackage{enumerate,enumitem}
\usepackage{mathrsfs}
\usepackage{bbm}
\usepackage[sc]{mathpazo}
\usepackage{eulervm,tikz}
\usetikzlibrary{positioning, shapes.geometric}
\usepackage{color}
\usepackage{tikz-cd}
\newtheorem{thm}{Theorem}
\newtheorem{prop}{Proposition}[section]
\newtheorem{lem}[prop]{Lemma}
\theoremstyle{definition}
\newtheorem{defi}[prop]{Definition}
\newtheorem{notation}[prop]{Notation}
\newtheorem{cor}[prop]{Corollary}
\newtheorem{rem}[prop]{Remark}
\newtheorem{ex}[prop]{Example}

\newcommand{\kr}{\mathbbm{k}}

\def\rk{\operatorname{rk}}
\def\stab{\operatorname{stab}}

\def\lcover{{\; <\!\!\! \cdot \;}}

\def\IS{\mathcal{I}_{\mathfrak S}}
\def\PS{\mathcal{P}_{\mathfrak S}}

\def\supp{\operatorname{supp}}
\def\Supp{\operatorname{Supp}}

\renewcommand{\check}[1]{{\overline{#1}}}
\newcommand{\sz}[1]{\overline{\overline{#1}}}
\newcommand{\downset}[1]{\lceil #1 \rceil}

\def\SSS{\mathscr S}
\def\II{\mathcal I}

\def\stella{{\normalfont ({\bf $\star$})}}

\begin{document}
\author{Alessio D'Al\`i}
\address[Alessio D'Al\`i]{Max-Planck-Institut f\"ur Mathematik in den Naturwissenschaften, Inselstra{\ss}e 22, 04103 Leipzig, Germany}
\curraddr{Mathematics Institute, University of Warwick, Coventry CV4 7AL, United Kingdom}
\email{Alessio.D-Ali@warwick.ac.uk}
\author{Emanuele Delucchi}
\address[Emanuele Delucchi]{D\'epartement de Math\'ematiques, Universit\'e de Fribourg, Chemin du Mus\'ee 23, 1700 Fribourg, Switzerland}
\email{emanuele.delucchi@unifr.ch}
\title[Stanley--Reisner rings for symmetric simplicial complexes]{
Stanley--Reisner rings for symmetric simplicial complexes, G-semimatroids and Abelian arrangements
}
\subjclass[2010]{13F55 (primary), 05E18, 13A50, 52C35, 55U10, 06A11 (secondary)}
\maketitle

\begin{abstract}
We extend the notion of face rings of simplicial complexes and simplicial posets to the case of finite-length (possibly infinite) simplicial posets with a group action. The action on the complex induces an action on the face ring, and we prove that the ring of invariants is isomorphic to the face ring of the quotient simplicial poset under a mild condition on the group action. We also identify a class of actions on simplicial complexes that preserve the homotopical Cohen--Macaulay property under quotients. 

When the acted-upon poset is the independence complex of a semimatroid, the $h$-polynomial of the ring of invariants can be read off the Tutte polynomial of the associated group action. Moreover, in this case an additional condition on the action ensures that the quotient poset is Cohen--Macaulay in characteristic $0$ and every characteristic that does not divide an explicitly computable number. This implies the same property for the associated Stanley--Reisner rings. In particular, this holds for independence posets and rings associated to toric, elliptic and, more generally, $(p,q)$-arrangements. As a byproduct, we prove that posets of connected components (also known as posets of {layers}) of such arrangements are Cohen--Macaulay with the same condition on the characteristic. 
\end{abstract}

\section{Introduction}

\subsection{Background}
A classical construction associates a commutative ring, called {\em Stanley-Reisner ring}, to every finite simplicial complex.
In the wake of pioneering work in the 1970s by R.\ Stanley, M.\ Hochster and G.\ Reisner, a rich research activity has blossomed around this bridge between combinatorics and topology on the one side and commutative algebra on the other, leading to major advances such as Stanley's proof of the Upper Bound Conjecture \cite{StanleyUBC}.

A recurring theme in this research area is to investigate properties of the class of Stanley-Reisner rings associated to special (combinatorially defined) families of simplicial complexes. 
A good example  is given by simplicial complexes that arise as the complex of independent sets of a matroid 
(``matroid complexes'' \cite{BjoAltro}). Such complexes are defined by abstract properties modeled on the family of linearly independent subsets of a given collection of vectors in a vector space, and the associated Stanley-Reisner rings have attracted a large body of work. A topic of particular interest is a sequence of integers related to the Hilbert series of the defining ideal, namely the coefficients of the ring's $h$-polynomial, 
which is strongly related to classical polynomial invariants of matroids. For instance, Stanley-Reisner rings of matroids are {\em Cohen-Macaulay} \cite{BjoAltro}, which implies positivity of said coefficients. The study of the properties of such integer sequences, especially as it relates to concavity properties, 
is a topical and very active field \cite{AHK} in which many questions and conjectures remain open to date.

{\em Simplicial posets} are a generalization of posets of faces of simplicial complexes  
(see Definition \ref{df_simplicialposet}). Stanley \cite{StaSP} defined a ``face ring'' associated to any finite simplicial poset which, in the special case of posets of faces of simplicial complexes, is isomorphic to the classical Stanley-Reisner ring. 

\subsection{Motivation and context}
The study of symmetries in the form of group actions on simplicial complexes has classical roots \cite{Bredon} and came into the focus of growing interest over the last years.  Significant results have been obtained in the combinatorial study of algebraically defined objects \cite{ElSk,3S} as well as in using symmetries in order to advance in combinatorial problems \cite{Lutz,McMS,Novik}, a special mention being deserved by the impact of the study of group actions on topological combinatorics \cite{AAF,Jiri}.  

Moreover, as we will discuss below in more detail, recent developments in the theory of arrangements lead to the study of structural aspects of group actions on matroids and posets. 
A peculiarity of the latter setup is that it does not meet the standard finiteness (or compactness) assumptions on which most of the extant literature relies (here, to the above-mentioned references we add some specific literature on group actions on posets, e.g., \cite{StaG,ThWe}).

It is then natural to wonder about the algebraic implications of group actions on complexes or posets in terms of the associated Stanley-Reisner rings. In fact, this line of research has been pursued in the literature \cite{GaSt,ReinerThesis,StaInv} but, again, always under finiteness conditions.
In particular, Victor Reiner proved that when a finite group $G$ acts on a balanced simplicial poset $P$ preserving labels then, for every integer $k$ not dividing the cardinality of $G$, if $P$ is Cohen--Macaulay in characteristic $k$ then so is the quotient poset $P/G$ \cite[Theorem 2.3.2]{ReinerThesis}. Here we study the preservation of the Cohen-Macaulay property in the case of (possibly infinite) groups acting on (possibly infinite) complexes without balancing conditions.

\subsection{Aim and results} We propose an enrichment of the Stanley-Reisner theory by considering group actions on finite-dimensional (but possibly infinite) simplicial complexes. In fact, in this context we find that the natural framework  is that of finite-length simplicial posets. We associate a face ring $\mathcal R(P)$ to each such simplicial poset $P$ (Definition \ref{df:anello}) and, given an action of a group $G$ on the poset, we study the ring $\mathcal R(P)^G$ of invariants of the induced action on the ring. We characterize precisely the group actions for which the quotient poset $P/G$ is again simplicial: these turn out to be the type of actions called {\em translative} in \cite{DR} (Lemma \ref{lem:simplicialquotient}). We prove that, given a translative action of a group $G$ on a simplicial poset $P$, the ring of invariants $\mathcal R(P)^G$ is isomorphic to the ring $\mathcal R (P/G)$ associated to the quotient poset (Theorem \ref{thm:invariant}).

We introduce a class of group actions we call {\em decoupled}
(a condition strictly stronger than translativity, see Definition \ref{df:decoupled}) and we prove that quotients of posets of faces of finite-dimensional homotopy Cohen--Macaulay simplicial complexes are again homotopy Cohen-Macaulay when the action is decoupled and the group is abelian (Theorem \ref{thm:hcm}). 

Then we turn to the matroidal case,  generalizing some of the properties of Stanley--Reisner rings of matroids to the case of semimatroids with group actions.
 We obtain that, if $P$ is the poset of independent sets of a semimatroid and the group action is {\em refined} (i.e., satisfies a condition stronger than translativity but that does not imply decoupling), then the (finite) poset $P/G$ (and the associated ring) is Cohen--Macaulay in characteristic $0$ and every characteristic not dividing an explicitly computable number $\delta$ (Theorem \ref{thm_CMSM}). Such restrictions on the characteristic arise because we rely on a classical lemma by Bredon \cite{Bredon}, but they are not artefacts of the proof: see Example \ref{RE:B} and \cite[Section 8]{PP}.
 
 Moreover, the characteristic polynomial of $P/G$ and  the $h$-polynomial of the ring $\mathcal R(P/G)$ are evaluations of the Tutte polynomial associated to any translative action on a semimatroid \cite[\S 3.4]{DR}.

As a byproduct, we prove that the quotient of any rank-finite geometric semilattice with respect to a translative and refined action is Cohen--Macaulay with the same conditions on the characteristic as the associated poset of independent sets (Theorem \ref{thm:PGCM}).

\subsection{Application: Abelian arrangements}\label{ss:AA1} Many aspects of the classical theory of arrangements of hyperplanes are being extended to encompass {\em toric arrangements} and {\em elliptic arrangements}. The aim is a general topological and combinatorial theory of {\em Abelian arrangements}. In the following we give a quick primer in this subject and refer to Section \ref{sec:AA} for more. 

An Abelian arrangement is a finite set $\mathscr A$ of level sets of group homomorphisms $\mathbb G^d\to \mathbb G$, where $\mathbb G$ is a complex algebraic group of dimension one.

In this context, a main combinatorial invariant is the {\em poset of layers}, i.e., the set
\begin{equation}\label{eq:PLA}
    \mathcal C (\mathscr A) := \{\textrm{conn. comp. of } \cap \mathscr X \mid \mathscr X\subseteq \mathscr A\}
\end{equation}
of connected components of intersections of subsets of $\mathcal A$, partially ordered by reverse inclusion \cite{dCP2,Zaslav}. 
There is as of yet little understanding of the structure of such posets  beyond linear arrangements, except from the case of  Weyl-type arrangements where the posets are known to be shellable \cite{DeGiPa} based on the explicit description given by Bibby \cite{Bibby2}. 

In the case of hyperplanes ($\mathbb G = \mathbb C$) this poset has the structure of a geometric lattice, and is equivalent to the arrangement's {\em matroid} data. 
The case of toric arrangements ($\mathbb G = \mathbb C^*$) has recently been in the focus of a considerable amount of research that was at first motivated by applications to commutative algebra \cite{BPS} and partition functions \cite{dCP2}, but recently gained momentum as an independent topic. Research on topological \cite{dACDMP,dCP1,Pagaria} and combinatorial \cite{ERS,Law,Moc1} aspects of toric (and elliptic, $\mathbb G=\mathbb E$, e.g, \cite{Bibby}) arrangements reaffirmed the importance of the poset $\mathcal C (\mathscr A)$.

In particular, the theory of {\em arithmetic matroids} \cite{BM,dAM} was developed as a combinatorial framework for toric arrangements, but the poset structure is not described by the arithmetic matroid (nor by an even more refined invariant, the matroid over $\mathbb Z$ \cite{FM}): Pagaria \cite{Pagaria2} constructed two central toric arrangements with non-isomorphic posets of layers but isomorphic arithmetic matroid (resp.\ matroid over $\mathbb Z$).

An attempt at a structural characterization of posets of layers of Abelian arrangements (that distinguishes the examples of \cite{Pagaria2}) has been carried out in \cite{DR} along the following lines (see Section \ref{sec:AA} for a more precise treatment).

The universal covering space of $\mathbb G^d$ is $\mathbb C^d$, and under the universal covering morphism the arrangement $\mathscr A$ lifts to a periodic arrangement $\mathscr A ^\upharpoonright$ of affine hyperplanes. An affine arrangement such as $\mathscr A ^\upharpoonright$ is customarily described by the associated {\em semimatroid} \cite{Ard,DR,Kawa} or, equivalently, by the poset $\mathcal C (\mathscr A^\upharpoonright)$ which in this case is naturally a {\em geometric semilattice}. The periodicity group acts naturally on this poset, and the quotient poset is isomorphic to $\mathcal C (\mathscr A)$ \cite[Remark 2.3]{DR}. The approach of \cite{DR}, then, is to study group actions on semimatroids (or, equivalently, on geometric semilattices) and to view the quotients of such actions as the natural framework for an abstract combinatorial theory of posets of layers of Abelian arrangements.
 In this context, our results imply the following.

    \begin{itemize}
        \item To every Abelian arrangement is naturally associated a Stanley--Reisner ring via the associated periodic semimatroid. 
        This ring is Cohen--Macaulay in characteristic $0$ and every characteristic not dividing an explicitly computable number $\delta_{\mathscr A}$, and its $h$-polynomial is an evaluation of the action's Tutte polynomial.
    \item The poset of layers of every Abelian arrangement is Cohen--Macaulay in characteristic $0$ and every characteristic not dividing $\delta_{\mathscr A}$, and its (rational) homotopy type is determined by the Tutte polynomial of the associated action.
         In the special case of hyperplane arrangements, this recovers the classical theory. In the case of (central) toric arrangements our rings are isomorphic to those studied by Martino \cite{Martino} and Lenz \cite{Lenz}, and the action's Tutte polynomial is Moci's arithmetic Tutte polynomial \cite{Moc1}.
\end{itemize}

We obtain similar results also in the even broader context given by {\em $(p,q)$-arrangements}, part of a class of arrangements in Abelian Lie groups studied by Liu, Tran and Yoshinaga \cite{LTY}. 

\subsection{Structure of the paper} We will start by laying out some basic material on group actions on posets (Section \ref{sec_basics}) and on complexes and posets (Section \ref{sec_sicosipo}), as well as recalling some background on topological and algebraic aspects (Section \ref{sec_toal}). In Section \ref{sec:CFSP} we define Stanley--Reisner rings for general finite-length simplicial posets and  prove  that we recover the classical theory in the case of trivial actions on finite posets. 
The naturality of translative actions with respect to taking invariant rings, resp.\ poset quotients is discussed in Section \ref{sec:IA}, while Section \ref{sec:RACM} presents our result about preservation of Cohen-Macaulayness under decoupled actions.

Then we turn to the matroidal case in Section \ref{sec:matroids}, where we introduce refined group actions on (semi)matroids and prove our results on the Cohen--Macaulay property for quotients of posets of independent sets and of flats. The application to the case of Abelian arrangements and $(p,q)$-arrangements is discussed in Section \ref{sec:AA}, after a quick review of the context.

\subsection{Acknowledgements} We would like to thank  Aldo Conca, Florian Frick, Roberto Pagaria, Giovanni Paolini and Tim R\"omer for their substantive feedback in early stages of our work. We are grateful to Christos Athanasiadis, Ben Blum-Smith, Jan Draisma, Matthias Lenz and Victor Reiner for helpful observations on the first arXiv version. We thank an anonymous referee for pointing out a gap in the previous version, which led us to the present, corrected statement of the results in sections \ref{sec:RACM}-\ref{sec:AA}.

The research we report on  began during Alessio D'Al\`i's stay at the University of Fribourg, which was supported by a Swiss Government Excellence Scholarship. 
Emanuele Delucchi was supported by the Swiss National Science Foundation professorship grant PP00P2\_150552/1; he also acknowledges the friendly hospitality of the Max Planck Institute for Mathematics in the Sciences in Leipzig.

\section{Group actions on posets}\label{sec_basics}

This section is devoted to recalling some basics about posets and laying some groundwork for the rest of the article. As a reference, we can point to the book \cite{Aigner} for a treatment of finite-length posets as well as the standard reference (for the finite case) by Stanley \cite{Sta}.

\subsection{Generalities on posets} 

A partially ordered set, for short \emph{poset}, is a set $P$ with a partial order relation $\leq$ (i.e., a reflexive, antisymmetric and transitive binary relation). As usual, $x<y$ means $x\leq y, x\neq y$. We write $x\lcover y$ if $x < y$ and $x\leq z < y$ implies $z=x$ (in this case we say that ``$y$ {\em covers} $x$''). We often only mention $P$ when the order relation is understood. 
A morphism of posets is an order-preserving function; it is an isomorphism if it has an order-preserving inverse. 

\begin{ex}\label{boolean}
Let $n\in \mathbb N$ and let $B_n$ denote the poset of all subsets of $\{1,2,\ldots,n\}$ partially ordered by inclusion. A {\em Boolean algebra on $n$ elements} is any poset isomorphic to $B_n$. 
\end{ex}

Let $\operatorname{Aut}(P)$ denote the set of {\em automorphisms} of $P$, i.e., of all isomorphisms from $P$ to itself;  $\operatorname{Aut}(P)$ is a group with respect to composition of functions.

\begin{defi}\label{df:poset}
An {\em action} $G\circlearrowright P$ of a group $G$ on a poset $P$ is a group homomorphism $G\to \operatorname{Aut}(P)$. As is customary, we write $gp$ for the image of any $p\in P$ under the automorphism associated to $g\in G$. 

We define the {\em quotient} $P/G$ to be the set of all orbits of elements of $P$ with a binary relation $Gp\leq Gq$ if $gp\leq q$ for some $g\in G$.
\end{defi}

\begin{rem}\label{rem_qp}$\,$\begin{itemize}
\item[(i)]
The binary relation defined on $P/G$ is always reflexive and transitive, but in general it might fail to be antisymmetric. 

\item[(ii)] If $P/G$ is a poset then the canonical ``quotient'' map $P\to P/G$, $p\mapsto Gp$ is a well-defined order-preserving map.
\end{itemize}
\end{rem}

A {\em chain} in a poset $P$ is any totally ordered subset $X\subseteq P$. The length of a chain $X$ is the cardinal number $\vert X \vert - 1$. 
The {\em length} of the poset $P$ is the maximum length  of a chain in $P$. The length of $P$ is denoted $\ell(P)$ and in general is allowed to be infinite.
The poset is called {\em of finite length} if $\ell(P)<\infty$. 

\begin{lem}\label{lem:cf_qp}
Let $P$ be a finite-length poset. Then, for every action $G\circlearrowright P$ the set $P/G$ with the binary relation of Definition \ref{df:poset} is a partially ordered set.
\end{lem}

\begin{proof}
By contraposition: as noted in Remark \ref{rem_qp}.(i), the only way in which $P/G$ can fail to be a poset is that there are $p,q\in P$ and $g,h\in G$ such that $gp \leq q$ and $hq \leq p$ but $Gp\neq Gq$. If $gp=q$ or $hq=p$ then $Gp=Gq$, thus it must be $gp < q$ and $hq < p$. But then, $... g^{-1}q > p > hq > hgp > hghq ...$ is an infinite chain in $P$. 
\end{proof}

Given a poset $P$ and an element $x\in P$ let $P_{\leq x}:=\{p\in P \mid p\leq x\}$ and consider it as a poset with the partial order induced from $P$. Given $A\subseteq P$ let $P_{\leq A}:=\bigcap_{a\in A}P_{\leq a}$ be the set of {\em lower bounds} of $A$. We define $P_{\geq x}$ and $P_{\geq A}$, the set of upper bounds, analogously. A {\em (lower) order ideal} (or {\em down-set}) of a poset $P$ is a subset $\mathfrak a \subseteq P$ such that $x\in \mathfrak a$ and $y\leq x$ implies $y\in \mathfrak a$. Examples of lower order ideals include subsets of the type $P_{\leq x}$, which we call {\em principal} lower order ideals (generated by $x$). {\em Upper order ideals}, resp.\ up-sets, are defined accordingly.

The {\em (closed) interval} between two elements $x,y\in P$ is the set $[x,y]:=P_{\geq x} \cap P_{\leq y}$ with the induced partial order. The corresponding {\em open interval} is $(x,y):=[x,y]\setminus \{x, y\}$. 

If a poset $P$ has a unique minimal element, this element is commonly denoted by $\hat 0$. Then $\{\hat 0\}= P_{\leq P}$ and we say that $P$ is {\em bounded below}. Analogously, $P$ is bounded above if it has a unique maximal element, usually denoted by $\hat 1$. 
We will often have to modify a poset by adding or removing extremal elements, and thus we introduce the following notation. 

\begin{notation}\label{not:posets}
Given a bounded-below poset $P$,
\begin{itemize}[parsep=0em,itemindent=-3em]
    \item[] $\check{P}:=P\setminus \{\hat0\}$ denotes the poset obtained by removing the minimal element;
    \item[] $\sz{P}$ denotes the poset obtained from by removing both $\hat 0$ and $\hat 1$ (if the latter exists);
    \item[] $\widehat{P}$ denotes the poset obtained from $P$ by adding a maximal element $\hat 1$.
\end{itemize}
 Moreover, $A(P)$ denotes the set of {\em atoms} of $P$, i.e., all $p\in P$ with $\hat 0 \lcover p$.
 \end{notation}

\begin{defi}\label{df_rank}
We call a bounded-below  poset {\em graded} if it possesses a {\em rank function}, i.e., a function $\rk\colon P\to \mathbb N$  such that $\rk(\hat 0)=0 $ and $\rk(y)=\rk(x)+1$ whenever $x\lcover y$. If such a rank function exists, then it is uniquely determined by $\rk(x)=\ell(P_{\leq x})$.

\end{defi}

To every graded, bounded-below poset $P$ of finite length $d$ is associated a {\em characteristic polynomial}

$$
\chi_{P}(t) := \sum_{x\in  P} \mu_P(\hat 0, x) t^{d-\rk(x)},
$$

where $\mu_P$ denotes the M\"obius function of $P$, see \cite[\S 3.7]{Sta}.

\begin{lem}\label{rm:spq} 
Let $P$ be a graded poset (possibly of infinite length). Then, for every action $G\circlearrowright P$ the set $P/G$ with the binary relation of Definition \ref{df:poset} is a partially ordered set.
\end{lem}
\begin{proof}
Automorphisms of graded posets preserve the rank of elements. If $P/G$ were not a poset, then, as in the proof of Lemma \ref{rem_qp}, we would find elements $p,q\in P$ and $g,h\in G$ with $ gp<q<h^{-1}p$, in particular $hgp<p$, implying that the rank of $hgp$ is strictly smaller than that of $p$ -- a contradiction. 
\end{proof}

\subsection{Translative actions}

We introduce a class of actions on posets that has been studied in \cite{DR} as a natural abstraction of the action induced by linear translations on the poset of intersections of a periodic hyperplane arrangement, whence the name (see Example \ref{rem:TAC} below). We refer to Section \ref{sec:AA} for a more precise discussion of this context.

\begin{defi}\label{def:pta}
An action $G\circlearrowright P$ is called {\em translative} if, for every $p\in P$ and $g\in G$, whenever the set $\{p,gp\}$ has an upper bound (i.e., if $P_{\geq\{p,gp\}}\neq \emptyset$) then $p=gp$.
\end{defi}

\begin{ex}\label{rem:TAC}
If $\mathscr A$ is any set of affine subspaces of a vector space and $G$ is a group of translations that permutes the elements of $\mathscr A$, then the induced action of $G$ on the poset $\mathcal C (\mathscr A)$ (defined as in Equation \eqref{eq:PLA}) is translative. In order to see this, consider an affine subspace $X\in \mathcal C(\mathscr A)$ and a translation $g\in G$ such that the set $\{X, gX\}$ is bounded above. This means that the intersection of the subspaces $X$ and $gX$ is nonempty -- but since translated subspaces are parallel, it must be $X=gX$.
\end{ex}

\begin{rem} \label{rem_ta}
If $G\circlearrowright P$ is a translative action, then 
\begin{itemize}
\item[(i)]
the intersection of any $G$-orbit $X\in P/G$ with any lower interval $P_{\leq p}$ consists of at most one element;
\item[(ii)] if $y\geq x$, then $\stab(y)\subseteq \stab(x)$ (since $g\in \stab(y)$ implies $x,gx\leq y$). 
\end{itemize}
\end{rem}

\begin{rem}\label{rem:scwol}
In particular, translative actions are related to actions on {\em scwols} in the sense of \cite[Chapter III.C, Definition 1.11]{BrHa} as follows. A translative action on a finite length poset $P$ induces an action on the {scwol} defined on the set $P$ by putting an arrow $x\to y$ whenever $x\leq y$.
\end{rem}

\def\quot{\varphi}
\begin{lem}\label{lem:low}
Let $G$ be a group acting translatively on a poset $P$ and suppose that $P/G$ is a poset. Then, for every $p\in P$ the restriction
$$
\quot_p: P_{\leq p} \to (P/G)_{\leq Gp}
$$
of the quotient map is an isomorphism of posets.
\end{lem}
\begin{proof} 
By Remark \ref{rem_qp} the function $\quot_p$ is well-defined and order-preserving. The definition of the ordering among orbits implies that every $X\leq Gp$ contains a representative $x\in X$, $x\leq p$. Then, Remark \ref{rem_ta} shows that the function $$\psi: (P/G)_{\leq Gp}\to P_{\leq p},\quad X\mapsto X\cap P_{\leq p}$$
is well-defined. To see that it is order-preserving consider $X\leq Y \leq Gp$ in $P/G$ and notice that $X=G\psi(X)$ and $Y=G\psi(Y)$. Then, $X\leq Y$ implies that there is $g\in G$ s.t. $g\psi(X)\leq \psi(Y)$. In particular $g\psi(X) \leq p$, and thus with Remark \ref{rem_ta} we have $g\psi(X)=\psi(X)$: we conclude $\psi(X)\leq \psi(Y)$ as required.

We are left with proving that $\psi$ and $\varphi_p$ are inverses. For every $X\in (P/G)_{\leq Gp}$ we have $\varphi_p\circ \psi (X) = G\psi(X) = X$, thus $\varphi_p\circ\psi = \operatorname{id}_{(P/G)_{\leq Gp}}$. Moreover, for every $q\leq p$ we have $Gq\cap P_{\leq p}\supseteq \{q\}$ and by Remark $\ref{rem_ta}$ this inclusion is an equality. Hence we compute $\psi\circ\varphi_p(q) = \psi(Gq) = q $. Thus, $\psi\circ\varphi_p = \operatorname{id}_{P_{\leq p}}$ as required.
 \end{proof}

\begin{lem}\label{lem:up}
Let $P$ be a poset, consider a translative action $G\circlearrowright P$ such that $P/G$ is a poset. Let $f\colon P \to P/G$ denote the quotient map as above. Then for every $X\in P/G$
\begin{itemize}
\item[(i)] ${\displaystyle f^{-1}((P/G)_{\geq X}) = \coprod_{x\in f^{-1}(X)} P_{\geq x}}$.
\end{itemize}
Moreover, for every $x\in P$ the following hold. \begin{itemize}
\item[(ii)] There is an isomorphism of posets
$$
(P/G)_{\geq Gx} \simeq P_{\geq x} / \operatorname{stab}(x)
$$
and the action $\stab(x)\circlearrowright P_{\geq x}$ is translative.
\item[(iii)] If $\stab(x)$ is normal in $G$ we can consider the group $H:=G/\stab(x)$. Then, $H$ acts translatively on $P_{\leq Gx}=f^{-1}((P/G)_{\leq Gx})$, and
$$
(P/G)_{\leq Gx} \simeq P_{\leq Gx} / H. 
$$
\end{itemize}
\end{lem}
\begin{proof}
Part (i) follows immediately by translativity (see Remark \ref{rem_ta}). 
For part (ii) write $X=Gx$ and compare the definitions:
$$P_{\geq x}/{\stab (x)} = \{\stab(x)y \mid y \geq_P x\}
,\quad 
(P/G)_{\geq X} = \{ Gy \mid y\geq_{P}gx \textrm{ for some }g\}$$

Now consider the map

$$P_{\geq x}/{\stab (x)} \to (P/G)_{\geq X},\quad\quad\stab(x)y\mapsto Gy.$$ It is clearly well-defined and order-preserving. We will provide an order-preserving inverse. Consider $Y\in (P/G)_{\geq X}$. By definition, there is 
$y\in Y$ such that $y\geq_{P} x$. By part (i) this $y$ is unique up to the action of $\operatorname{stab}(x)$. Thus the function
 $$(P/G)_{\geq X} \to P_{\geq x}/{\stab (x)},\quad\quad Y\mapsto \stab(x)y$$ is well-defined.
 A straightforward check proves that this function is also order-preserving and it is indeed inverse to the previous. Translativity of $\stab(x)\circlearrowright P_{\geq x}$ is also easily verified.

Let us now consider part (iii). By definition of the quotient poset, every $Y\in (P/G)_{\leq Gx}$ is of the form $Y=Gy_Y$ for some $y_Y\leq x$. Translativity of $G\circlearrowright P$ implies uniqueness of such a $y_Y$, thus we have defined an order-preserving map  $\varphi: (P/G)_{\leq Gx} \to P_{\leq Gx}/H$, $Y\mapsto H y_Y$. A straightforward check shows that the obvious order-preserving map $\psi: P_{\leq Gx}/H \to (P/G)_{\leq Gx}$, $Hy \mapsto Gy$ is inverse to $\varphi$. 
An easy check of the definition verifies the translativity claim and concludes the proof.
\end{proof}

We conclude with a remark that will allow us to "split" actions of abelian groups.
\todo{Any G, H normal?}
\begin{lem}\label{lem:quac}
Let $G$ be an abelian group acting on a ranked poset $P$. Then, for every subgroup $H$
\begin{itemize}
\item[(a)] $G/H$ acts on $P/H$, and $(P/H)/(G/H)= P/G$;
\item[(b)] If the action $G\circlearrowright P$ is translative, then so is the action of $G/H$ on $P/H$.
\end{itemize}
\end{lem}
\begin{proof}$\,$
\begin{itemize}
\item[(a)] That $G/H$ acts on $P/H$ is easy to check, using abelianity of $G$. 
Moreover, the orbit of any $Hp\in P/H$ under $G/H$ is $\{(g+H)\cdot hp\mid p\in P, h\in H, g\in G \}=Gp$ equals the orbit of $p$ under $G$. 
\item[(b)] If $Hy \geq (g+H) Hx $ and $Hy \geq Hx$ for some $x,y\in P$ and some $g\in G$, then we can choose elements $h,h'\in H$ with $hx\leq y$ and $(h'+g)x\leq y$, and both $x$ and $(h'-h+g)x$ are below $(-h)y$ in $P$, whence $(h'-h+g)x=x$ by translativity of $G\circlearrowright P$ and thus $(g+H) Hx = Hx$, proving translativity of the action of $G/H$.
\end{itemize}

\end{proof}

\subsection{Refined actions}\label{df:refined}

Let $P$ be a graded poset of (finite) length $d$, and let $G$ be a finitely generated free Abelian group.
Suppose that $G$ acts on $P$ so that there is some $k\in \mathbb N$ satisfying
\begin{center}
\stella{} for all $x\in P$,  $\stab (x)$ is a direct summand of $G$ of rank $k(d-\rk(x))$,
\end{center}
where $\rk$ is the poset's rank function (see Definition \ref{df_rank}).

\begin{rem}\label{gr_ref}$\,$
\begin{itemize}
\item[(i)]
Since $G$ is a finitely generated free Abelian group, the condition for a subgroup $H$ of $G$ to be a direct summand of $G$ is equivalent to $H$ being a {\em pure} subgroup, meaning that $G/H$ has no torsion elements (equivalently, $nh \in H$ implies $h\in H$ for every $h\in G$ and every $n>0$). See \cite[\S 16A]{CurtisReiner}.

\item[(ii)] For every $x\in P$ of maximal rank, \stella{} implies $\stab(x)=\{0\}$. Moreover,
\stella{} implies also that $G\simeq \mathbb Z^{kd}$.
\end{itemize}
\end{rem}

\begin{defi}[Refined actions] 
We call a group action on a graded poset $P$ {\em refined} if it is translative and satisfies \stella{} for some $k\in \mathbb N$. If we wish to specify the number $k$, we will call the action {\em $k$-refined}.
\end{defi}

\begin{lem}\label{lem:sg}$\,$ 
Suppose that the action of $G$ on $P$ is $k$-refined for some $k\in\mathbb N$.
\begin{itemize}
\item[(i)] For every $x\in P$ the action of the group $\stab(x)$ on $P_{\geq x}$ is $k$-refined.
\item[(ii)]  For every $x\in P$ the action of the group $G/\stab(x)$ on $P_{\leq Gx}$ is $k$-refined.
\end{itemize}
\end{lem}
\begin{proof} By Lemma \ref{lem:up} we immediately know that both actions $\stab(x) \circlearrowright P_{\geq x}$ and $G/\stab(x) \circlearrowright P_{\leq Gx}$ are translative. Thus we only have to check condition \stella.

We start with (i). First, notice that $P_{\geq x}$ is ranked of rank $d':=d-\rk(x)$. Call $\rk_{\geq x}$ the rank function of $P_{\geq x}$.
Call $G':= \stab(x)$ and consider $y\in P_{\geq x}$. By Remark \ref{rem_ta}.(ii), translativity of the action implies that $\stab_G(y) \subseteq G'$. Hence $\stab_{G'}(y)=\stab_{G}(y)$ and, by assumption, this group has rank $k(d-\rk(y))=k(d-\rk_{\geq x}(y)-\rk(x)) = k(d'-\rk_{\geq x}(y))$ as required. Moreover, recall that $G'$ is a direct summand of $G$, say $G=G'\oplus H$. In particular,  $G'$ is free Abelian.
 Finally, every torsion element in $G'/\stab_{G'}(y) =\stab_{G}(x)/\stab_{G}(y)$ is a torsion element in $G/\stab_{G}(y)=(G'\oplus H) /\stab_{G}(y) = G'/\stab_{G'}(y) \oplus H. $ Since by assumption $\stab_{G}(y)$ is pure in $G$, we conclude that $\stab_{G'}(y)$ is pure in $G'$.

Now let us turn to (ii). Notice that $P_{\leq Gx}$ is ranked of length $d'':=\rk(x)$, and the rank function $\rk_{\leq Gx}$ is the restriction of the rank function $\rk$ of $P$. 
Since the original action satisfies \stella{} we can write $G=H\oplus \stab(x)$ for some subgroup $H$, so $G/\stab(x)\simeq H$ is free Abelian.
Now fix $y\in P_{\leq Gx}$ and consider $\stab_H(y)$. By definition there is $g\in G$ with $gx\geq y$. Hence, by commutativity of $G$ and with Remark \ref{rem_ta}.(ii), $\stab_G(x)=\stab_G(gx) \subseteq \stab_G(y)$. Thus
\begin{equation}\tag{$\dagger$}\label{eq:quotH}
\stab_H(y) \simeq \stab_G(y)/\stab_G(x).
\end{equation} 
From this  we can prove purity of $\stab_H(y)$  as a subgroup of $H$ by writing 
 \begin{equation*}\label{eq:isolem}
 H/\stab_H(y) \simeq (G/\stab_G(x)) / (\stab_G(y)/\stab_G(x)) \simeq G/\stab_G(y)
\end{equation*}
 and noticing the latter group is torsion-free by assumption. 
 Moreover, using Equation \eqref{eq:quotH} and property \stella{} for $G\circlearrowright P$ we can compute the rank of $\stab_H(y)$ to be $(d-\rk(y) ) - (d-\rk(x)) = \rk(x)-\rk(y) = d'' - \rk_{\leq Gx}(y)$, as required.

\end{proof}

\section{Simplicial structures}\label{sec_sicosipo}

\subsection{Simplicial posets} \label{sec:simplicial_posets}

\begin{defi}[Compare {\cite{Sta}}]\label{df_simplicialposet}
A finite-length, countable poset $P$ is called {\em simplicial} if it has a unique minimal element, and for all $p\in P$ the lower interval $P_{\leq p}$ is a Boolean algebra. 
\end{defi}

\begin{figure}[h]
\centering
\begin{tikzpicture}
\node (empty) at (0,0) {$\varnothing$};
\node (a) at (-1.5, 1) {$a$};
\node (b) at (0,1) {$b$};
\node (c) at (1.5, 1) {$c$};
\node (l1) at (-1.5, 2) {${\ell}_1$};
\node (l3) at (0, 2) {${\ell}_3$};
\node (l2) at (1.5, 2) {${\ell}_2$};
\node (l4) at (3, 2) {${\ell}_4$};
\node (T1) at (-0.75, 3) {$T_1$};
\node (T2) at (0.75, 3) {$T_2$};
\draw (empty) -- (a) -- (l1) -- (T1);
\draw (a) -- (l3) -- (T2);
\draw (empty) -- (b) -- (l4);
\draw (empty) -- (c) -- (l4);
\draw (l2) -- (T2);
\draw[preaction={draw=white, -,line width=6pt}] (T1) -- (l2);
\draw[preaction={draw=white, -,line width=6pt}] (l1) -- (T2);
\draw[preaction={draw=white, -,line width=6pt}] (l3) -- (T1);
\draw[preaction={draw=white, -,line width=6pt}] (l3) -- (c) -- (l2);
\draw[preaction={draw=white, -,line width=6pt}] (l1) -- (b) -- (l2);
\end{tikzpicture}
\caption{A finite simplicial poset $P$ that is not the poset of faces of a simplicial complex.}
\label{simplicialposet_example}
\end{figure}
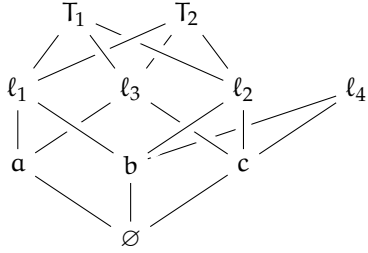

\begin{rem}
Every simplicial poset is graded in the sense of Definition \ref{df_rank}. In particular, if $P$ is a simplicial poset and $p\in P$, then $\rk(p)$ equals the number of elements of the Boolean algebra $P_{\leq p}$, i.e., $P_{\leq p} $ is isomorphic to $ B_{\rk(p)}$ (cf.~Example \ref{boolean}). 
\end{rem}

\begin{defi}\label{df:fh}

We also recall from \cite{StaSP} the definition of the $f$-vector of a finite simplicial poset $P$ of length $d$, 
$$
f(P):= (f_{-1}(P),\ldots, f_{d-1}(P)),\quad \textrm{ where }f_i(P):=\vert\{x\in P \mid \operatorname{rk}(x)=i+1\}\vert, 
$$
and of the associated $h$-polynomial
$$
h_P(t):= t^d \sum_{i=0}^{d} f_{i-1}(P)\left(\frac{1-t}{t}\right)^{d-i},
$$
where it is customary to set $f_{-1}(P) = 1$ for every $P$.

\end{defi}

\begin{rem}\label{rem:chih}
The $h$- and the characteristic polynomial of a simplicial poset $P$ are related as follows:
$$
\chi_P(t) = \sum_{i=0}^d f_{i-1}(P) (-1)^i t^{d-i} = (-1)^d h_P\left(\frac{1}{1-t}\right).
$$
\end{rem}

The following lemma expresses the fact that focussing on simplicial posets is natural when considering translative group actions.

\begin{lem} \label{lem:simplicialquotient}
Let $G$ be a group acting on a simplicial poset $P$. Then $P/G$ is a simplicial poset if and only if the action is translative.
\end{lem}
\begin{proof}
That quotients of simplicial posets by translative actions are simplicial is an immediate consequence of Lemma \ref{rm:spq} and Lemma \ref{lem:low}.

For the reverse implication, consider a group $G$ acting by automorphisms on a poset $P$ and suppose that $P/G$ is a simplicial poset.   Given  $y\in P$, since automorphisms preserve poset rank, we know that the rank of $y$ in $P$ equals the rank of $Gy$  in $P/G$. Simpliciality of $P$ and $P/G$ then implies that $\vert P_{\leq y}\vert = \vert (P/G)_{\leq Gy} \vert$.

 By way of contradiction suppose now that the action is not translative. This means that we can choose $y$ so that there are $x \in P$, $g\in G$ with $x<y$, $gx<y$ and $gx\neq x$.  
In particular, the quotient map $P_{\leq y} \to P/G_{\leq Gy}$ is not injective. Since this map is surjective by definition, we conclude $\vert P_{\leq y}\vert > \vert (P/G)_{\leq Gy} \vert$ -- a contradiction.
\end{proof}

\subsection{Simplicial complexes}
 Let $V$ be a set. An abstract simplicial complex on the vertex set $V$ is a family $\Sigma$ of finite subsets of $V$ that is closed under taking subsets (i.e., $\sigma \in \Sigma$ and $\tau\subseteq \sigma$ implies $\tau\in \Sigma$). We will assume that every one-element subset of $V$ is in $\Sigma$.

\begin{ex}
To every finite-length poset $P$ is canonically associated a simplicial complex $\Delta(P)$, called the {\em order complex} of $P$. The set of vertices of $\Delta(P)$ is the set $P$ and a subset of $P$ defines a simplex if and only if it is totally ordered.
\end{ex}

Elements of $\Sigma$ are called {\em faces} (or {\em simplices}), and every face $\sigma\in \Sigma$ has a dimension $\dim\sigma := \vert \sigma \vert -1$. The dimension of $\Sigma$ is then the maximum of the dimensions of its faces: this can be an infinite cardinal, and we call $\Sigma$ finite-dimensional if its dimension is finite. It is customary to call $\Sigma$ {\em pure} (or pure-dimensional) if all maximal faces of $\Sigma$ have the same dimension.

The set $\Sigma$ partially ordered by inclusion is a simplicial poset $P_\Sigma$, ranked by cardinality of its elements. The atoms of $P_\Sigma$ correspond to the single-element subsets of $V$. Every action $G\circlearrowright P_\Sigma$ induces an action on $V$ and, vice-versa, the action on the whole poset is determined by the action on the vertices.

\begin{rem}\label{rem_traeq} Given a translative action $G\circlearrowright P_\Sigma$, for every $x\in P_\Sigma$ we have $\stab(x)=\bigcap_{v\in x}\stab(v)$. (In fact, translativity implies $\stab(x)\subseteq \stab(v)$ for every $v\in x$, see Remark \ref{rem_ta}.(ii). Since $\Sigma$ is a simplicial complex, every $g\in\bigcap_{v\in x}\stab(v) $ fixes $x$ because it fixes all its vertices.)

\end{rem}

In particular, from Remark \ref{rem_traeq} we conclude that for every translative action $G\circlearrowright P_\Sigma$ the associated action on $\Sigma$ satisfies Bredon's condition (A), see \cite[\S III.1]{Bredon}. 

\begin{rem}
Let $G\circlearrowright \Sigma$ be a group action on a simplicial complex. In general, the set of orbits of simplices does not have the structure of a simplicial complex, but only a natural partial order  $P_{\Sigma}/G$. If this partially ordered set is in fact of the form $P_{\Sigma'}$ for some simplicial complex $\Sigma'$, then we call $\Sigma'$ the quotient of $\Sigma$ and write $\Sigma/G$.
\end{rem}

\begin{lem}\label{lem_bredon}
Let $G\circlearrowright P$ be a translative group action on a finite-length simplicial poset. Then the induced action on $\Delta(P)$ is {\em regular} in the sense of Bredon \cite[\S III.1]{Bredon}. In particular, there is an isomorphism of simplicial complexes $\Delta(P)/G \simeq \Delta(P/G)$.
\end{lem}
\begin{proof}
The first part of the claim is based on the following property: 
\begin{center}
(\#$_G$) if both $p_1 < p_2 <\ldots < p_k$ and $g_1p_1 < g_2p_2 <\ldots < g_kp_k$ are simplices of $\Delta(P)$ then there is some  $g\in G$ with $gp_i=g_ip_i$ for all $i$. 
\end{center}
We prove this for translative group actions by induction on $k$ following \cite[Proposition 1.1]{Bredon}. The claim is trivial for $k=1$. Let then $k>1$, by induction hypothesis we can assume that there is $g\in G$ with $gp_i=g_ip_i$ for all $i<k$. This means that  $p_1<p_2<\ldots <p_{k-1} < g^{-1}g_k p_k$. Then both $p_{k-1}$ and $g^{-1}g_kp_{k-1}$ are below $g^{-1}g_kp_k$ thus by translativity $g^{-1}g_kp_{k-1}=p_{k-1}$. Then $g^{-1}g_k$ fixes all of $p_1<\ldots<p_{k-1}$, and therefore $g_ip_i=gp_i=g_kp_i$ for all $i<k$. Then, $g_k$ is the desired element with $g_kp_i=g_ip_i$ for all $i\leq k$. Now $G\circlearrowright \Delta(P)$ is regular if (\#$_H$) holds for the action of every subgroup $H$ of $G$. But the action of every subgroup of a group acting translatively is also translative, thus the first part of the claim is proved.

Now for the second part of the claim notice first that regularity of an action implies that the quotient simplicial complex is well-defined \cite[p. 117]{Bredon}. 
Consider the natural order-preserving map
$$
P_{\Delta(P)}/G\to P_{\Delta(P/G)},\quad G\{p_1<\ldots <p_k\} \mapsto Gp_1 <\ldots < Gp_k.
$$
Moreover, given $Gp_1 <\ldots < Gp_k$ in $P_{\Delta(P/G)}$ we can choose representatives of the orbits $g_1p_1<g_2p_2<\ldots < g_kp_k$ and, with (\#$_G$), we find $g\in G$ with $gp_i=g_i$ for all $i$. Hence, $\{Gp_1, \ldots , Gp_k\}=G\{p_1,\ldots,p_k\}$ and the (order-preserving) map
$$
P_{\Delta(P/G)}\to P_{\Delta(P)}/G,\quad
Gp_1 <\ldots < Gp_k\mapsto G(p_1< \ldots <p_k)
 $$
 is well-defined and inverse to the previous one, giving an isomorphism of posets  $P_{\Delta(P)}/G\simeq P_{\Delta(P/G)}$. In particular, since $\Delta(P/G)$ is by definition a simplicial complex, the quotient simpicial complex $\Delta(P)/G$ is well-defined and isomorphic to $\Delta (P/G)$.
\end{proof}

\begin{rem}
Notice that Bredon's definition of quotient simplicial complex agrees with ours in the sense that $P_{\Sigma/G}=P_{\Sigma}/G$ whenever the action is regular. Therefore, the previous lemma is a slight extension of the discussion in \cite[p.117]{Bredon}. 
\end{rem}

\begin{rem}\label{rem:oiss}
Given $\sigma \in P_{\Sigma}$, then the poset $P_{\geq \sigma}$ is again the poset of faces of a simplicial complex. More precisely, consider the set $V':=\{\sigma' \in P_{\Sigma} \mid \sigma \lcover \sigma'\}$. Then, $P_{\geq x}\simeq P_{\Sigma'}$ where $\Sigma'=\{ X\subseteq V' \mid \cup X\in \Sigma\}$, which is isomorphic to the {\em link} of $x$ in $\Sigma$. (Recall \cite[\S 9.9]{BjTopMet} that the {\em link}  of a face $\sigma$ in a simplicial complex $\Sigma$ is the simplicial complex $\operatorname{lk}(\sigma):=\{\tau\in \Sigma \mid \tau\cap \sigma = \emptyset \textrm{ and } \tau\cup \sigma \in \Sigma\}$.) 
\end{rem}

\section{Topological and algebraic aspects}\label{sec_toal}

\subsection{Topology} Every abstract simplicial complex as defined in the previous section has a geometric realization \cite[Chapter 1, \S 2]{Munkres} which is unique up to homeomorphism. Hence, every abstract simplicial complex has a well-defined homotopy type.

Moreover, to every partially ordered set $P$ we can associate the abstract simplicial complex of all finite chains in $P$. This is called the {\em order complex} of $P$ (notice that its dimension equals the length of $P$). 
Thus a well-defined homotopy type can be associated to every partially ordered set. Order-preserving maps induce simplicial maps of order complexes and, thus, continuous maps between geometric realizations. When we will discuss topological attributes of a poset we will always think of them as referred to the order complex. For instance, with $H_i(P)$, $\pi_i(P)$ etc. we will mean the homology or homotopy groups of the order complex. For a more careful introduction and a broader account of the scope of combinatorial algebraic topology see, e.g., \cite{BjTopMet,Koz}.

\subsubsection{Connectivity}
Given an integer $t \in \mathbb N$ we call a simplicial complex $\Sigma$  {\em $t$-connected} if it is nonempty, connected and the homotopy groups $\pi_i(\Sigma)$ are trivial for all $i=1,\ldots,t$. Analogously we call $\Sigma$ {\em $t$-acyclic} over a ring $R$ if the reduced homology $\widetilde{H}_i(\Sigma,R)$ is trivial for $i=0,\ldots, t$. We extend these definition by saying any nonempty $P$ to be ``$(-1)$-acyclic'' and ``$(-1)$-connected''.

If a $d$-dimensional complex $\Sigma$ is $(d-1)$-connected or $(d-1)$-acyclic, then we will say that it is ``well-connected'' or ``acyclic through codimension $1$''.

\def\CS{\Sigma}
\begin{rem}[On shellability]\label{remshell} A simplicial complex $\CS$ is called {\em shellable} if there exists a well-ordering $\prec$ on its set $\mathcal M$ of maximal simplices  so that for all $\sigma\in \mathcal M \setminus \min_{\prec}\mathcal M$, the intersection of $\sigma$ with the subcomplex $\CS_{\prec\sigma}$ induced by the simplices in $\mathcal M_{\prec \sigma}$ is a pure simplicial complex of dimension $\dim \sigma -1$.

An alternative, operationally advantageous characterization is the following: for all $m_1,m_2\in \mathcal M$ such that $m_1\prec m_2$ and for any $\tau\subset m_1\cap m_2$, there is $m_3\in \mathcal M$ with $m_3\prec m_2$ and $x\in m_2$ such that $m_1\cap m_2\subseteq m_3\cap m_2 = m_2\setminus \{x\}$, see \cite[Lemma 2.3]{BjWa1}.

If a pure simplicial complex $\CS$ of dimension $d$ is shellable, then it has the homotopy type of a wedge of $(d-1)$-dimensional spheres, and thus it is well-connected (see \cite[Theorem 4.1]{BjWa1}).

\end{rem}

We state for later reference  the following lemma, summarizing results by Mirzaii and van der Kallen and by Bj\"orner, Wachs and Welker. 

\begin{lem}[cf.\ Theorem 3.8 of \cite{MirKa} and Corollary 3.2 of {\cite{BWW}}]\label{LBWW}
Let $f: P\to Q$ be a poset map. Fix $t\geq 0$ and suppose that for all $q\in Q$ 
\begin{itemize}
\item[(1)] $Q_{>q}$ is $(t - \ell(Q_{<q}) - 2)$-connected, and 
\item[(2)] the fiber $f^{-1}(Q_{\leq q})$ is $\ell(Q_{<q})$-connected.
\end{itemize}
Then the homotopy groups of $P$ and $Q$ agree up to (and including) degree $t$. In particular, $P$ is $t$-connected if and only if $Q$ is $t$-connected.

The same holds, when $P$ and $Q$ are finite, replacing $t$-connectivity by $t$-aciclicity over any given field $\mathbb K$.
\end{lem}

\begin{rem}\label{wellconn}
If $P$ and $Q$ are ranked posets of the same length $d$ and $f$ is rank-preserving, then Lemma \ref{LBWW} with $t=d-1$ can be stated as follows. {\em Suppose that, for all $q\in Q$,  (1) $Q_{>q}$ is well-connected and (2) $f^{-1}(Q_{\leq q})$ is well-connected. Then $Q$ is well-connected if and only if $P$ is.} In order to see that this claim reduces to Lemma \ref{LBWW} under the stated conditions, notice that $(t-\ell(Q_{<q}) -2)=(d-\ell(Q_{<q}) -3)=(\ell(Q_{> q}) -1)$ and that $\ell(Q_{<q})=\ell(Q_{\leq q})-1=\ell(f^{-1}(Q_{\leq q}))-1$ (the last equality because $f$ is rank-preserving).
 \end{rem}

\subsubsection{Cohen-Macaulay complexes and posets}
We will be concerned with a well-known property of simplicial complexes with strong commutative-algebraic implications. %
\begin{defi} \label{def:CM}
We will call a simplicial complex $\Sigma$ of dimension $d$ {\em Cohen-Macaulay} if for every face $\sigma \in \Sigma$ (including the case $\sigma=\emptyset$) the link $\operatorname{lk}(\sigma)$ of $\sigma$ in $\Sigma$ is $(\dim(\operatorname{lk}(\sigma))-1)$-connected. 
  For any given ring of coefficients $R$, we call $\Sigma$ ``Cohen-Macaulay over $R$'' if $\operatorname{lk}(\sigma)$ is $(\dim(\operatorname{lk}(\sigma))-1)$-acyclic over $A$. For every  $k\in \mathbb N$, we will say 
\begin{center}
``$\Sigma$ is $CM(k)$''
\end{center}
to signify that $\Sigma$ is Cohen-Macaulay over every ring whose characteristic is either $0$ or does not divide $k$.
As usual, we can define the Cohen-Macaulay properties and notations for any poset $P$ by reference to the associated order complex $\Delta(P)$.
\end{defi}

\begin{rem}\label{PCM_crit}
We recall from  \cite[\S 11.9]{BjTopMet} the following characterization of Cohen-Macaulay posets. If $P$ is a poset of finite length, then $P$ is homotopy Cohen-Macaulay (resp.\ Cohen-Macaulay in a given characteristic $k$) if and only if every open interval $(x,y)\subseteq \widehat{P}$ is $(\ell(x,y)-1)$-connected (resp.\ $(\ell(x,y)-1)$-acyclic).
\end{rem}

\subsubsection{Euler characteristic}

As a last piece of preparation let us consider Euler characteristics of posets.  
We let $\epsilon(P)$ denote the {\em reduced} Euler characteristic of the order complex of $P$, which can be expressed for instance by the alternating sum $\sum_{i\geq 0} (-1)^i\dim\widetilde{H}_i(\Delta(P),\mathbb Q)$ of the dimensions of the rational homology groups of $\Delta(P)$.  (From this, the standard Euler characteristic an be recovered by adding $1$, see \cite{Sta}.) We give for completeness a proof of the following elementary lemma.

\begin{lem}\label{lem:EuCar}
Let $P$ be a bounded-below poset. Then
$$
\epsilon(\check{P}) = - \chi_P(1).
$$
If $P$ is also bounded above, then
$$
\epsilon(\sz{P})= \chi_P(0).
$$
\end{lem}

\begin{proof}
Key is the following interpretation of the M\"obius function of a bounded poset $P$ known as ``Hall's theorem'' \cite[Proposition 3.8.5]{Sta}:
$$
\mu_P(\hat 0 , \hat 1) = \epsilon(\sz{P}).
$$
Recall that, since $P$ is bounded below, $\widehat{P}$ denotes the poset $P$ with a unique maximal element $\hat 1$ adjoined. Then 
$\epsilon (\check{ P}) = \epsilon \left(\sz{\widehat{P}}\right) = \mu_{\widehat{P}}(\hat 0 , \hat 1)$.

On the other hand, by definition of the M\"obius function \cite[Chapter 3, \S 7]{Sta}
$$
\mu_{\widehat{P}}(\hat 0 ,\hat 1) = 
- \sum_{\hat 0 \leq x < \hat 1} \mu_{\widehat{P}}(\hat 0, x) = 
- \sum_{x\in P} \mu_{{P}}(\hat 0, x) = - \chi_P(1).
$$

We conclude that $\epsilon (\check{P}) = - \chi_P(1)$.

If $P$ is also bounded above, then immediately $\chi_P(0)=\mu_P(\hat 0, \hat 1)= \epsilon(\sz{P})$. 
\end{proof}

\subsection{Algebra}
Given a finite simplicial complex $\Sigma$ and a field $\kr$, consider the polynomial ring $\kr[x_v \mid v\in V]$ whose variables are indexed by vertices of $\Sigma$. Therein define the ideal
$$
\mathcal I _{\Sigma} := 
( 
\prod_{v\in \sigma} x_v \mid \sigma \not\in \Sigma
)
$$
generated by all monomials corresponding to non-faces of $\Sigma$.

\begin{defi}
The {\em Stanley-Reisner ring} of a finite simplicial complex $\Sigma$ is the quotient
$$
\mathcal R(\Sigma) := \kr [x_v \mid v\in V]/ \mathcal I _{\Sigma}
$$
\end{defi}

\begin{rem}
One of the basic facts about Stanley-Reisner rings is that the Cohen-Macaulay property for $\Sigma$ (see Definition \ref{def:CM}) implies  the (algebraic) Cohen-Macaulay property for the ring $\mathcal R(\Sigma)$ over every field  (the latter algebraic property is in fact equivalent to a homological version of the Cohen-Macaulay property for $\Sigma$, obtained by replacing ``connected'' by ``acyclic'' in Definition \ref{def:CM}, see \cite{Reisner}). 
\end{rem}

\begin{rem} Stanley defined an analogous ring associated to every finite simplicial poset. We will review this definition in Section \ref{sec:StanleyFin}.
\end{rem}

\section{Stanley-Reisner rings of finite-length simplicial posets}
\label{sec:CFSP}

Throughout the section, $P$ will be a simplicial poset with atoms $A(P)$ and $G$ will be a group acting on $P$ by automorphisms. Recall that, under these hypotheses, the quotient $P/G$ is again a simplicial poset. Let us denote by $f\colon P \to P/G$ the standard projection. Given $p \in P$, as usual we use the notation $Gp$ to denote $f(p) \in P/G$.

\subsection{The definition}
Let $\mathrm{max}(P)$ denote the set of  maximal elements of $P$. Given a nonempty collection $\tau \subseteq \mathrm{max}(P)$, we denote by $\downset{\tau}$ the order ideal of $P$ given by $\bigcap_{t \in \tau}P_{\leqslant t}$. Note that $\downset{\tau}$ is the poset of faces of a simplicial complex. The associated Stanley-Reisner ring will be denoted by $R_{\downset{\tau}}$. When $\downset{\tau} = \{\hat{0}\}$, one has that $R_{\downset{\tau}} = \kr$. For any set bounded above inside a given $\downset{\tau}$, the join in $\downset{\tau}$ is well-defined.

Following Yuzvinsky \cite{Yuzv}, we call $X(P)$ the set 
\[
\{
\downset{\tau} \mid \emptyset \neq \tau \subseteq \max(P)
\}
\]
of all (lower) order ideals $\downset{\tau} \subseteq P$ coming from nonempty collections $\tau \subseteq \mathrm{max}(P)$, with the partial order given by \[\downset{\tau} \leqslant_{X(P)} \downset{\tau'} \text{ if and only if } \downset{\tau} \supseteq \downset{\tau'}.\]

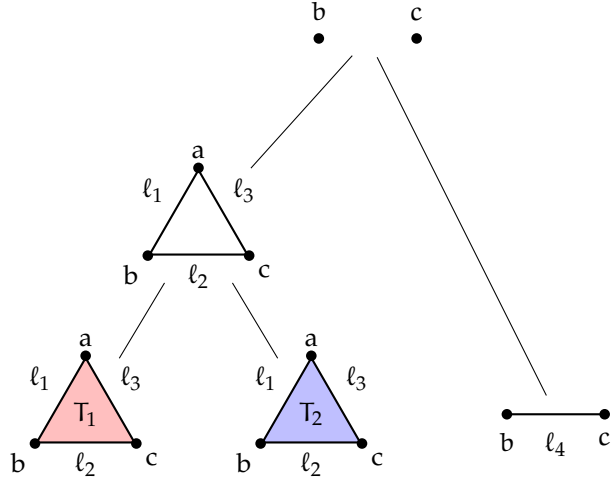
\begin{figure}[h]
\centering
\begin{tikzpicture}
\node[draw, thick, fill=red!25, minimum size=1.5cm, regular polygon, 
    regular polygon sides=3,
    label=corner 1:a, label=corner 2:b, label=corner 3:c, label=side 1:${\ell}_1$, label=side 2:${\ell}_2$, label=side 3:${\ell}_3$, label=center:$T_1$] (T1) at (-3, 0) {};
\node[draw, thick, fill=blue!25, minimum size=1.5cm, regular polygon, 
    regular polygon sides=3,
    label=corner 1:a, label=corner 2:b, label=corner 3:c, label=side 1:${\ell}_1$, label=side 2:${\ell}_2$, label=side 3:${\ell}_3$, label=center:$T_2$] (T2) at (0, 0) {};
\node[draw, thick, minimum size=1.5cm, regular polygon, 
    regular polygon sides=3,
    label=corner 1:a, label=corner 2:b, label=corner 3:c, label=side 1:${\ell}_1$, label=side 2:${\ell}_2$, label=side 3:${\ell}_3$] (T1T2) at (-1.5, 2.5) {};
\node[label=above: $b$] (b) at (0.1,5) {};
\node (bc) at (0.75,5) {};
\node[label=above: $c$] (c) at (1.4,5) {};
\node[label=below: $b$] (lowb) at (2.6,0) {};
\node[label=below: $\ell_4$] (l4) at (3.25,0) {};
\node[label=below: $c$] (lowc) at (3.9,0) {}; 
\fill (T1.corner 1) circle[radius=2pt];
\fill (T1.corner 2) circle[radius=2pt];
\fill (T1.corner 3) circle[radius=2pt];
\fill (T2.corner 1) circle[radius=2pt];
\fill (T2.corner 2) circle[radius=2pt];
\fill (T2.corner 3) circle[radius=2pt];
\fill (T1T2.corner 1) circle[radius=2pt];
\fill (T1T2.corner 2) circle[radius=2pt];
\fill (T1T2.corner 3) circle[radius=2pt];
\fill (b) circle[radius=2pt];
\fill (c) circle[radius=2pt];
\fill (lowb) circle[radius=2pt];
\fill (lowc) circle[radius=2pt];
\draw[thick, shorten <= -4pt, shorten >= -4pt] (lowb) -- (lowc);
\draw[shorten <= 12pt, shorten >= 12pt] (T1) to (T1T2);
\draw[shorten <= 12pt, shorten >= 12pt] (T2) to (T1T2);
\draw[shorten <= 4pt, shorten >= 4pt] (l4) to (bc);
\draw[shorten <= 4pt, shorten >= 18pt] (bc) to (T1T2);
\end{tikzpicture}
\caption{The poset $X(P)$ associated with the simplicial poset $P$ from Figure \ref{simplicialposet_example} in Section \ref{sec:simplicial_posets}. Since each order ideal of $P$ of  the form $\downset{\tau}$ is the poset of faces of a simplicial complex, we use here such (labeled) simplicial complexes for visualization purposes. For instance, $\downset{\{T_1, T_2\}} = P_{\leqslant{T_1}} \cap P_{\leqslant{T_2}}$ is represented above by the empty triangle with vertices $\{a, b, c\}$ and edges $\{\ell_1, \ell_2, \ell_3\}$, whereas $\downset{\{T_1, {\ell}_4\}} = \downset{\{T_2, {\ell}_4\}} = \downset{\{T_1, T_2, {\ell}_4\}}$ is represented by the disconnected simplicial complex with two points $b$ and $c$.}
\label{XP_example}
\end{figure}

Any poset can be made into a topological space by considering the Alexandrov topology, where the open sets are the upper sets of the poset. Any sheaf (say, of commutative rings) on a poset is then completely determined by the assignment of a covariant functor from the poset (seen as a category as in Remark \ref{rem:scwol}) to the category of commutative rings (see, e.g., \cite{KeBa}).

With this in mind, again following \cite{Yuzv}, we define the sheaf of commutative rings $Y(P)$ on the poset $X(P)$ by the assignments
\begin{align*} X(P) \ni \downset{\tau} &\mapsto R_{\downset{\tau}}\\
(\downset{\tau} \leqslant_{X(P)} \downset{\tau'}) &\mapsto \pi^{\downset{\tau}}_{\downset{\tau'}}\colon R_{\downset{\tau}} \twoheadrightarrow R_{\downset{\tau'}},
\end{align*}
where $\pi^{\downset{\tau}}_{\downset{\tau'}}\colon R_{\downset{\tau}} \twoheadrightarrow R_{\downset{\tau'}}$ is the natural projection.

\begin{defi} \label{df:anello}
The {\em Stanley-Reisner ring} of $P$ is then the ring of (global) sections
$$
\mathcal R(P) := \Gamma(Y(P)).
$$
We view any $q\in \mathcal R(P)$ as an $X(P)$-tuple of polynomials, and for every $\downset{\tau}\in X(P)$ we denote by $q_{\downset{\tau}}\in R_{\downset{ \tau}}$ the component associated to $\downset{\tau}$.
\end{defi}

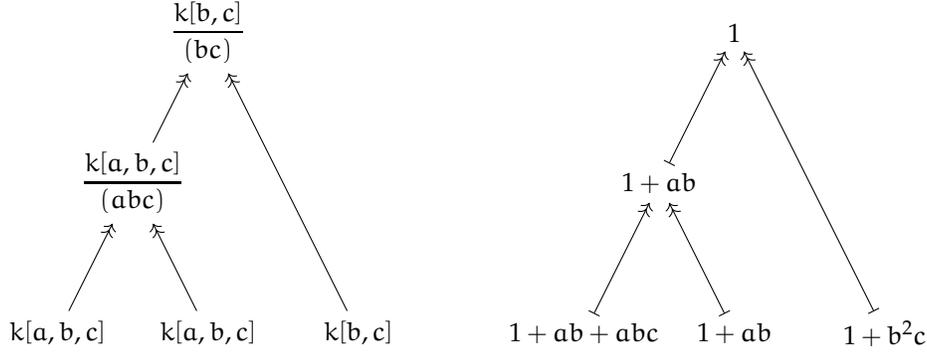
\begin{figure}[h]
\centering
\begin{tikzpicture}
\node (T1) at (0,0) {$k[a,b,c]$};
\node (T2) at (2,0) {$k[a,b,c]$};
\node (l4) at (4,0) {$k[b,c]$};
\node (T1T2) at (1,2) {$\displaystyle\frac{k[a,b,c]}{(abc)}$};
\node (T1T2l4) at (2,4){$\displaystyle\frac{k[b,c]}{(bc)}$};
\draw[->>] (T1) -- (T1T2);
\draw[->>] (T1T2) -- (T1T2l4);
\draw[->>] (T2) -- (T1T2);
\draw[->>] (l4) -- (T1T2l4);
\node (pT1) at (7,0) {$1+ab+abc$};
\node (pT2) at (9,0) {$1+ab$};
\node (pl4) at (11,0) {$1+b^2c$};
\node (pT1T2) at (8,2) {$1+ab$};
\node (pT1T2l4) at (9,4){$1$};
\draw[|->>] (pT1) -- (pT1T2);
\draw[|->>] (pT1T2) -- (pT1T2l4);
\draw[|->>] (pT2) -- (pT1T2);
\draw[|->>] (pl4) -- (pT1T2l4);
\end{tikzpicture}
\caption{On the left, the sheaf $Y(P)$ on the poset $X(P)$ from Figure \ref{XP_example}; on the right, one of the global sections of $Y(P)$. Data are here organized in the same shape as Figure \ref{XP_example}.}
\end{figure}

With a slight abuse of notation, we will reserve the letter $\sigma$ to denote both a maximal element of $P$ and a minimal element of $X(P)$, since every minimal element $\downset{\{\sigma\}}$ of $X(P)$ is uniquely determined by the choice of a maximal element $\sigma$ of $P$. In particular, we will write $q_\sigma$ for $q_{\downset{\{\sigma\}}}$.

Let us record here a simple observation that will come in handy in what follows.

\begin{defi}
Let $p \in P$. We say that a monomial $m = \prod_{v \in A(P)} x_v^{\alpha_v}$ is \emph{supported at $p$} if $\alpha_v = 0$ for all $v \in A(P) \setminus P_{\leqslant p}$.
\end{defi}

\begin{lem} \label{supportlemma}
Let $q \in \Gamma(Y(P))$, $p \in P$, $\sigma, \sigma' \in \mathrm{max}(P)$ such that $p \leqslant \sigma, \sigma'$. Then the monomials supported at $p$ appear with the same coefficients inside $q_{\sigma}$ and $q_{\sigma'}$.
\end{lem}
\begin{proof}
Since $p \in P_{\leqslant\sigma} \cap P_{\leqslant\sigma'}$, any monomial supported at $p$ appears inside $q_{\sigma}$ (respectively $q_{\sigma'}$) with the same coefficient it has inside $q_{\downset{\tau}} \in R_{\downset{\tau}}$, where $\tau = \{\sigma, \sigma'\}$.
\end{proof}

\subsection{The finite case}\label{sec:StanleyFin}

In this section we show that, in the finite case, we recover the classical constructions.

Fix a field $\kr$. Given a finite simplicial poset $P$, we let $\widetilde{S}$ be the polynomial ring $\kr[x_p \mid p \in P]$ (with the grading given by $\deg(x_p) = rk(p)$) and $S := \tilde{S}/(x_{\hat 0}-1)$ be its dehomogenization with respect to $x_{\hat 0}$.
Given an order ideal $\mathfrak{a}$ inside $P$ and $p, q \in \mathfrak{a}$, let $u_{\mathfrak{a}}(p,q)$ be the (possibly empty) set of minimal common upper bounds in $\mathfrak{a}$ of $p$ and $q$. In symbols, \[u_{\mathfrak{a}}(p,q) := \{z \in \mathfrak{a} \mid z \geqslant p, q \text{ and } \nexists z' \in \mathfrak{a} \text{ s.t. } z > z' \geq p, q\}.\]

We define $\widetilde{I_{\mathfrak{a}}^P}$ to be the following ideal of $\widetilde{S}$:
\[\widetilde{I_{\mathfrak{a}}^P} := (x_p \mid p \notin \mathfrak{a}) + (x_px_q - x_{p \wedge q}\sum_{z \in u_{\mathfrak{a}}(p,q)}x_z \mid p, q \in \mathfrak{a}, \ p, q \text{ incomparable}),\]
where the sum $\sum_{z \in u_{\mathfrak{a}}(p,q)}x_z$ is taken to be zero when $p$ and $q$ have no common upper bound in $\mathfrak{a}$ (in this case, the generator $x_px_q - x_{p \wedge q}\sum_{z \in u_{\mathfrak{a}}(p,q)}x_z$ is simply the monomial $x_px_q$).
Let then $I_{\mathfrak{a}}^P$ be the dehomogenization of $\widetilde{I_{\mathfrak{a}}^P}$ with respect to $x_{\hat 0}$, i.e. the ideal of $S$ obtained from $\widetilde{I_{\mathfrak{a}}^P}$ by setting to $1$ all occurrences of the variable $x_{\hat 0}$.

\begin{defi}[{\cite[Section 3]{StaSP}}] 
To every finite simplicial poset $P$ there are two associated rings,
$$\widetilde{A_P}:=\widetilde{S}/\widetilde{I_P^P},\quad\quad A_P:= S/{I_P^P}.$$
Note that $A_P$ is a finitely generated positively graded algebra with $(A_P)_0 = \kr$. 
\end{defi}

\begin{prop}[{\cite[Section 3]{StaSP}}] 
When $P$ is the poset of faces of a finite simplicial complex $\Sigma$, then $A_P$ coincides with $\mathcal R(\Sigma)$. In general, the algebraic $h$-polynomial of $A_P$ (i.e.~the numerator of the Hilbert series of $A_P$ expressed as a rational function) coincides with $h_P(t)$, the combinatorial $h$-polynomial of the simplicial poset $P$ (see Definition \ref{df:fh}).

If the poset $P$ is Cohen-Macaulay (in some characteristic), then so are $A_P$ and $\widetilde{A_P}$ (in the same characteristic).  
\end{prop}

Let us record here a technical observation that will come in handy while proving Proposition \ref{posetideals} below.

\begin{lem} \label{radical}
Let $P$ be a finite simplicial poset and $\mathfrak{a}$ a lower order ideal of $P$. Then $I_{\mathfrak{a}}^P$ is a radical ideal.
\end{lem}
\begin{proof}
By \cite[Proposition 16.23]{BrunsVetter} it is enough to check that $\widetilde{I_{\mathfrak{a}}^P}$ is radical, i.e. $\widetilde{S}/\widetilde{I_{\mathfrak{a}}^P}$ is reduced. Now note that $\widetilde{S}/\widetilde{I_{\mathfrak{a}}^P} \cong \widetilde{T} / \widetilde{I_{\mathfrak{a}}^{\mathfrak{a}}}$, where $\widetilde{T} = \kr[x_p \mid p \in \mathfrak{a}]$. By \cite[Lemma 3.4]{StaSP},  $\widetilde{T} / \widetilde{I_{\mathfrak{a}}^{\mathfrak{a}}}$ is an algebra with straightening law on $\mathfrak{a}$ (seen as a poset on its own) and hence is reduced, as desired.
\end{proof}

\begin{prop} \label{posetideals}
Let $P$ be a finite simplicial poset and let $\mathfrak{a}$, $\mathfrak{b}$ be order ideals of $P$. Then:
\begin{enumerate}
\item $I_{\mathfrak{a}}^P + I_{\mathfrak{b}}^P = I_{\mathfrak{a}\cap\mathfrak{b}}^P$ (and hence the correspondence $\mathfrak{a} \mapsto I_{\mathfrak{a}}^P$ reverses inclusions);
\item $I_{\mathfrak{a}}^P \cap I_{\mathfrak{b}}^P = I_{\mathfrak{a}\cup\mathfrak{b}}^P$.
\end{enumerate}
\end{prop}

\begin{proof}
For brevity's sake, given elements $p$ and $q$ in $P$, we will use the symbol $y_{p \wedge q}$ defined in the following way:
\[y_{p \wedge q} := \begin{cases}1 & \text{if } p \wedge q = \hat 0 \\
x_{p \wedge q} & \text{otherwise.}\end{cases}\]

We now proceed to the proof.

\begin{enumerate}
\item By definition, \[\begin{split}I_{\mathfrak{a}}^P + I_{\mathfrak{b}}^P &= (x_p \mid p \notin \mathfrak{a}) + (x_px_q - y_{p \wedge q}\sum_{z \in u_{\mathfrak{a}}(p,q)}x_z \mid p, q \in \mathfrak{a}, \ p, q \text{ incomparable})\\ &+ (x_p \mid p \notin \mathfrak{b}) + (x_px_q - y_{p \wedge q}\sum_{z \in u_{\mathfrak{b}}(p,q)}x_z \mid p, q \in \mathfrak{b}, \ p, q \text{ incomparable}) \\ &= (x_p \mid p \notin \mathfrak{a} \cap \mathfrak{b}) + (x_px_q - y_{p \wedge q}\sum_{z \in u_{\mathfrak{a}}(p,q)}x_z \mid p, q \in \mathfrak{a}, \ p, q \text{ incomparable}) \\ &+ (x_px_q - y_{p \wedge q}\sum_{z \in u_{\mathfrak{b}}(p,q)}x_z \mid p, q \in \mathfrak{b}, \ p, q \text{ incomparable}).\end{split}\]

Now pick two incomparable elements $p$, $q$ in $\mathfrak{a}$ and consider the generator $x_px_q - y_{p \wedge q}\sum_{z \in u_{\mathfrak{a}}(p,q)}x_z$.
\begin{itemize}
\item If at least one of $p$ and $q$ does not lie in $\mathfrak{b}$, then the whole generator is superfluous.
\item If both $p$ and $q$ lie in $\mathfrak{b}$, one rewrites the generator as
\[x_px_q - y_{p \wedge q}\sum_{z \in u_{\mathfrak{a} \cap \mathfrak{b}}(p,q)}x_z,\]
since any $x_z$ corresponding to an upper bound in $\mathfrak{a} \setminus \mathfrak{b}$ of $p$ and $q$ is superfluous.
\end{itemize}
The claim now follows.

\item It is enough to prove that $I_{\mathfrak{a}}^P = \bigcap_{q \in \mathfrak{a}}I_{(q)}^P$, where $(q)$ denotes the principal order ideal generated by the element $q$. Note that, by part (1), $I_{(q)}^P\supseteq I_{\mathfrak{a}}^P$ for any $q \in \mathfrak{a}$. Moreover, for $\mathfrak{a} = \varnothing$ the claim holds trivially (taking the empty intersection of ideals to be the ring $S$). Let then $\mathfrak{a}$ be nonempty.

Since by Lemma \ref{radical} $I_{\mathfrak{a}}^P$ is radical, it is enough to prove that, whenever a prime ideal $\wp$ in $S$ contains $I_{\mathfrak{a}}^P$, then it also contains the prime ideal $I_{(q)}^P$ for some $q \in \mathfrak{a}$.

To prove this, consider the (nonempty) set of maximal elements in $\mathfrak{a}$. There are two cases:
\begin{enumerate}
\item[(i)] There is exactly one maximal element $M$ in $\mathfrak{a}$. In this case, $\wp \supseteq I_{\mathfrak{a}}^P = I_{(M)}^P$ and we are done.
\item[(ii)] There are at least two maximal elements $M$, $M'$ in $\mathfrak{a}$. In this case the monomial $x_Mx_{M'}$ must belong to $I_{\mathfrak{a}}^P$ and hence the prime ideal $\wp$ is forced to contain at least one of $x_M$ and $x_{M'}$. Without loss of generality, say $x_M \in \wp$. Then $\wp \supseteq I_{\mathfrak{a}}^P + (x_M) = I_{{\mathfrak{a}}\setminus\{M\}}^P$. 
\end{enumerate}

Since $P$ is finite and case (ii) gives us a reduction from $\mathfrak{a}$ to a strictly smaller order ideal $\mathfrak{a} \setminus \{M\}$, we are bound to meet case (i) at some point. Notice that the element we eventually find in case (i) will not in general be maximal in the original order ideal $\mathfrak{a}$.

\end{enumerate}
\end{proof}

\begin{cor}
The ideals $I_{\mathfrak{a}}^P$ in $S$, ordered by inclusion, form a distributive lattice 
 with respect to sum and intersection.
\end{cor}
\begin{proof}
This is a direct consequence of Proposition \ref{posetideals}, since order ideals of $P$ form a distributive lattice with respect to intersection and union \cite[\S 3.4]{Sta}.
\end{proof}

\begin{thm}\label{thm:finitePL}
For every finite simplicial poset $P$,
$$
\mathcal R(P) = A_P,
$$
i.e., we recover Stanley's ring associated to $P$.
\end{thm}
\begin{proof}

Let  $\{p_1, \ldots, p_k\}$ denote the maximal elements of $P$. The sheaf $Y(P)$ on $X(P)$ satisfies the hypotheses of \cite[Example 3.3]{BBR} and hence we get a full description for the global sections of $Y(P)$:
\[
\Gamma(Y(P)) = S / \bigcap_{i=1}^k I_{(p_i)}^P = S / I_P^P,
\]
where the last equality comes from Proposition \ref{posetideals}. 
\end{proof}

\begin{rem}
After completing our work, we became aware of work of L\"u and Panov \cite{LP} from which our Theorem \ref{thm:finitePL}  follows. 
Moreover, notice that  Brun and R\"omer \cite{BrunRoemer} prove the analogous statement for $\widetilde{A_P}$. 

Notice that our proof also produces a minimal prime ideal decomposition of $I_P^P$ into the $I^P_{(p_i)}$.
\end{rem}


\section{Induced actions and invariant rings}\label{sec:IA}

We now want to bring group actions into the picture, proving that every group action on a simplicial poset induces an action on the associated ring. Moreover, if the action is translative we will prove that the ring of invariants is isomorphic to the ring associated to the quotient poset.

\subsection{The induced action on $\mathcal R (P)$} Consider an action $G\circlearrowright P$ of a group $G$ by automorphisms of $P$. Given $g \in G$, let us define $\omega^g$ as the automorphism of $\kr[x_v \mid v \in A(P)]$ obtained by sending $x_v$ into $x_{gv}$. Given a nonempty collection $\tau \subseteq \mathrm{max}(P)$, let $g\tau := \{gt \mid t \in \tau\}$. One has that the assignment $p \mapsto gp$ induces an isomorphism of (posets of faces of) simplicial complexes $\downset{\tau} \xrightarrow{\cong} \downset{g\tau}$. Hence, $\omega^g$ induces a ring isomorphism $\omega^g\colon R_{\downset{\tau}} \xrightarrow{\cong} R_{\downset{g\tau}}$ between the corresponding Stanley-Reisner rings. Moreover, if $\tau \leqslant_{X(P)} \tau'$, then the following diagram
\begin{equation}\label{goodrestriction}
\begin{tikzcd}
R_{\downset{\tau}} \arrow{r}{\omega^g} \arrow[swap]{d}{\pi^{\downset{\tau}}_{\downset{\tau'}}} & R_{\downset{g\tau}} \arrow{d}{\pi^{\downset{g\tau}}_{\downset{g\tau'}}} \\
R_{\downset{\tau'}} \arrow[swap]{r}{\omega^g} & R_{\downset{g\tau'}}
\end{tikzcd}
\end{equation}

commutes.

\begin{defi}
Consider a simplicial poset $P$ with an action of a group $G$. Given any element $q=(q_{\downset{\tau}})_{\downset{\tau}\in X(P)}$ of $\mathcal R(P)$ and any $g\in G$ define the $X(P)$-tuple $gq$ by
\[gq_{\downset{\tau}} := \omega^g(q_{\downset{g^{-1}\tau}}).\]
\end{defi}

\begin{lem}
For any given action of a group $G$ on a simplicial poset $P$ by automorphisms, the assignment
\[
G\times \mathcal R(P) \to \mathcal R(P),
\quad\quad (g,q) \mapsto gq
\]
defines an action of $G$ by (ring) automorphisms on $\mathcal R(P)$.
\end{lem}
\begin{proof}

Let us first check that $gq$ is indeed a global section of $Y(P)$: given $\tau \leqslant_{X(P)} \tau'$, one has that
 
\begin{align*}
\pi^{\downset{\tau}}_{\downset{\tau'}}(gq_{\downset{\tau}}) &= \pi^{\downset{\tau}}_{\downset{\tau'}} \circ \omega^g(q_{\downset{g^{-1}\tau}}) & \\
&= \omega^g \circ \pi^{\downset{g^{-1}\tau}}_{\downset{g^{-1}\tau'}}(q_{\downset{g^{-1}\tau}}) &\text{by the commutativity of \eqref{goodrestriction}}\\
&= \omega^g(q_{\downset{g^{-1}\tau'}}) &\text{since $q$ is a section}\\
&= gq_{\downset{\tau'}}. &
\end{align*}
Hence for every $g$ we have a well-defined map $\varphi_g\colon \mathcal R(P) \to \mathcal R(P)$, $q\mapsto gq$. That $\varphi_g$ is a ring homomorphism follows from the fact that $\omega^g$ is a ring homomorphism and that elements ($X(P)$-tuples) of $\mathcal R(P)$ are added and multiplied componentwise.

Finally, one checks that $G \to \operatorname{Aut}(\mathcal{R}(P))$, $g \mapsto \varphi_g$ is a group homomorphism as desired.

\end{proof}

\subsection{Invariant rings for translative actions} 
From now on we will require that the action of $G$ be translative. Consider $\sigma \in \max(P)$, $\Sigma \in \max(P/G)$ such that $G\sigma = \Sigma$. By Lemma \ref{lem:low}, translativity allows us to define the following ring isomorphism:

\[
\begin{tikzcd}
\hspace{35pt}\kr[x_v \mid v \in A(P) \cap P_{\leqslant \sigma}] \arrow[bend left=10]{rr}{\zeta^{\sigma}_{\Sigma}} & \cong & \kr[x_\mathcal{V} \mid \mathcal{V} \in A(P/G) \cap (P/G)_{\leqslant \Sigma}] \arrow[bend left=10]{ll}{\zeta^{\Sigma}_{\sigma}}
\end{tikzcd}
\]

For each $g \in G$, considering $\omega^g\colon \kr[x_v \mid v \in A(P) \cap P_{\leqslant \sigma}] \xrightarrow{\cong} \kr[x_v \mid v \in A(P) \cap P_{\leqslant g\sigma}] $ yields

\begin{eqnarray}\label{zeta_omega}
\zeta^{g\sigma}_{\Sigma} \circ \omega^g = \zeta^{\sigma}_{\Sigma} & \text{and} & \omega^g \circ \zeta^{\Sigma}_{\sigma} = \zeta^{\Sigma}_{g\sigma}.
\end{eqnarray}

We now have all the ingredients for the main result of the section.
\begin{thm} \label{thm:invariant}
Let $G$ be a group acting translatively on the simplicial poset $P$. Then there is a ring isomorphism $\mathcal R (P)^G \cong \mathcal R(P/G)$.
\end{thm}

\begin{rem}\label{rem:R} 
A result in the same vein as
 Theorem \ref{thm:invariant} was proved by Reiner \cite[Theorem 2.3.1]{ReinerThesis} for label-preserving group actions on finite balanced simplicial complexes. Such actions are translative in our sense, hence we recover Reiner's theorem in this case.
\end{rem}

\begin{proof} We will define two mutually inverse ring homomorphisms $\varphi$ and $\psi$ between $\mathcal R(P)^G=\Gamma(Y(P))^G $ and $ \mathcal R(P/G) = \Gamma(Y(P/G))$.

\smallskip
\noindent {\em 1. Definition of $\varphi$.}
Let $\varphi\colon \Gamma(Y(P))^G \to \Gamma(Y(P/G))$ be the map of rings defined for each $q \in \Gamma(Y(P))^G$ and $\downset{T} \in X(P/G)$ by

\[(\varphi(q))_{\downset{T}} := \pi^{\Sigma}_{\downset{T}} \circ \zeta^{\sigma}_{\Sigma}(q_{\sigma}),\]

where $\Sigma \in \mathrm{max}(P/G)$ is such that $\downset{T} \subseteq (P/G)_{\leq \Sigma}$ (in other words, $\Sigma$ is any minimal element of $X(P/G)$ lying below $\downset{T}$) and $\sigma \in f^{-1}(\Sigma)$.

\smallskip
\noindent {\em 2. The map $\varphi$ is well-defined.}
First of all, once $\Sigma$ is fixed, due to the $G$-invariance of $q$ it makes no difference which representative $\sigma$ we pick inside $f^{-1}(\Sigma)$. More precisely, one has that

\begin{align*}
\zeta^{g\sigma}_{\Sigma}(q_{g\sigma}) &= \zeta^{g\sigma}_{\Sigma} \circ \omega^g(q_{\sigma}) & \text{since $q$ is $G$-invariant}\\
&= \zeta^{\sigma}_{\Sigma}(q_{\sigma}) & \text{due to \eqref{zeta_omega}}. 
\end{align*}

Let us now check that the definition of $\varphi$ is independent on the choice of $\Sigma$. Let us pick $(\Sigma, \sigma), (\Sigma', \sigma')$ as in the definition above and let us show that 

\begin{equation}\label{varphi_welldef_claim}
\pi^{\Sigma}_{\downset{T}} \circ \zeta^{\sigma}_{\Sigma}(q_{\sigma}) = \pi^{\Sigma'}_{\downset{T}} \circ \zeta^{\sigma'}_{\Sigma'}(q_{\sigma'}).
\end{equation}

To do this, it is enough to check that any nonzero monomial in $R_{\downset{T}}$ appears with the same coefficient on both sides of \eqref{varphi_welldef_claim}. We will use the notation $\langle m, f \rangle$ to denote the coefficient of the monomial $m$ in the polynomial $f$. Let us fix a nonzero monomial $M$ in $R_{\downset{T}}$. By definition, $M$ must be supported at an element $\Upsilon \in \downset{T}$. By construction, one has that $\Upsilon \leqslant_{P/G} \Sigma$ and $\Upsilon \leqslant_{P/G} \Sigma'$. One can now choose $\upsilon  \in f^{-1}(\Upsilon)$ such that $\upsilon \leqslant_P \sigma$ and $\upsilon \leqslant_P g\sigma'$ for some $g \in G$. Note that, by construction, $\zeta^{\Sigma}_{\sigma}M$ and $\zeta^{\Sigma'}_{g\sigma'}M$ represent the same monomial (supported at $\upsilon$), which we will denote by $m$. By Lemma \ref{supportlemma} we then have that $\langle m, q_{\sigma}\rangle = \langle m, q_{g\sigma'}\rangle$. We now get the desired result, since

\begin{align*}
    \langle M, \pi^{\Sigma}_{\downset{T}} \circ \zeta^{\sigma}_{\Sigma}(q_{\sigma}) \rangle &= \langle M, \zeta^{\sigma}_{\Sigma}(q_{\sigma}) \rangle = \langle m, q_{\sigma} \rangle = \langle m, q_{g\sigma'} \rangle \\&= \langle M, \zeta^{g\sigma'}_{\Sigma'}(q_{g\sigma'})\rangle = \langle M, \zeta^{\sigma'}_{\Sigma'}(q_{\sigma'})\rangle = \langle M, \pi^{\Sigma'}_{\downset{T}} \circ \zeta^{\sigma'}_{\Sigma'}(q_{\sigma'})\rangle.
\end{align*}

Since the restriction maps behave well, one has that $\varphi(q)$ is indeed a global section of $Y(P/G)$ and thus $\varphi$ is well-defined. \hfill $\triangle$

\smallskip
\noindent {\em 3. Definition of $\psi$.}
Let $$\psi\colon \Gamma(Y(P/G)) \to \Gamma(Y(P))^G$$ be the map of rings defined for each $Q \in \Gamma(Y(P/G))$ and $\downset{\tau} \in X(P)$ by

\[(\psi(Q))_{\downset{\tau}} := \pi^{\sigma}_{\downset{\tau}} \circ \zeta^{G\sigma}_{\sigma}(Q_{G\sigma}),\]
where $\sigma \in \mathrm{max}(P)$ is such that $\downset{\tau} \subseteq P_{\leq \sigma}$.

\smallskip
\noindent {\em 4. The map $\psi$ is well-defined.}
We need to check that $\psi$ is independent on the choice of $\sigma$, i.e.~that, given $\sigma$ and $\sigma'$ as above,

\begin{equation} \label{psi_welldef_claim}
\pi^{\sigma}_{\downset{\tau}} \circ \zeta^{G\sigma}_{\sigma}(Q_{G\sigma}) = \pi^{\sigma'}_{\downset{\tau}} \circ \zeta^{G\sigma'}_{\sigma'}(Q_{G\sigma'}).
\end{equation}

Let us consider a nonzero monomial $m$ in $R_{\downset{\tau}}$ supported at $p \in \downset{\tau}$. Note that $\zeta^{\sigma}_{G\sigma}m$ and $\zeta^{\sigma'}_{G\sigma'}m$ represent the same monomial (supported at $Gp$), which we will denote by $M$. Since $Gp \leqslant_{P/G} G\sigma, G\sigma'$, by Lemma \ref{supportlemma} one has that $\langle M, Q_{G\sigma} \rangle = \langle M, Q_{G\sigma'}\rangle$. This leads us to the desired result, since

\begin{align*}
    \langle m, \pi^{\sigma}_{\downset{\tau}} \circ \zeta^{G\sigma}_{\sigma}(Q_{G\sigma})\rangle &= \langle m, \zeta^{G\sigma}_{\sigma}(Q_{G\sigma})\rangle = \langle M, Q_{G\sigma}\rangle = \langle M, Q_{G\sigma'} \rangle \\&= \langle m, \zeta^{G\sigma'}_{\sigma'}(Q_{G\sigma'})\rangle = \langle m, \pi^{\sigma'}_{\downset{\tau}} \circ \zeta^{G\sigma'}_{\sigma'}(Q_{G\sigma'})\rangle.
\end{align*}

Again, $\psi(Q)$ is a global section of $Y(P)$ since restriction maps behave well. We still need to check that $\psi(Q)$ is $G$-invariant. This is indeed the case, since for each $Q \in \Gamma(Y(P/G))$, $\downset{\tau} \in X(P)$, and $g \in G$ one has that

\begin{align*}
    g\psi(Q)_{\downset{g\tau}} &= \omega^g \circ \psi(Q)_{\downset{\tau}} & \\
    &=\omega^g \circ \pi^{\sigma}_{\downset{\tau}} \circ \zeta^{G\sigma}_{\sigma}(Q_{G\sigma}) & \\
    &= \pi^{g\sigma}_{\downset{g\tau}} \circ \omega^g \circ \zeta^{G\sigma}_{\sigma}(Q_{G\sigma}) & \text{by the commutativity of \eqref{goodrestriction}} \\
    &= \pi^{g\sigma}_{\downset{g\tau}} \circ \zeta^{G\sigma}_{g\sigma}(Q_{G\sigma}) & \text{due to \eqref{zeta_omega}} \\
    &= \psi(Q)_{\downset{g\tau}}. & 
\end{align*}

It follows that $\psi$ is well-defined. \hfill $\triangle$

\smallskip

\noindent {\em 5. $\varphi$ and $\psi$ are inverses.} Finally, it is easy to see that $\varphi$ and $\psi$ are inverse to each other. Given $q \in \Gamma(Y(P))^G$, for every $\downset{\tau} \in X(P)$ one has that

\begin{align*}
\psi(\varphi(q))_{\downset{\tau}} &= \pi^{\sigma}_{\downset{\tau}} \circ \zeta^{G\sigma}_{\sigma}(\varphi(q)_{G\sigma}) & \\ &= \pi^{\sigma}_{\downset{\tau}} \circ \zeta^{G\sigma}_{\sigma} \circ \pi^{G\sigma}_{G\sigma} \circ \zeta^{\sigma}_{G\sigma}(q_{\sigma}) & \text{choosing $\sigma$ inside $f^{-1}(G\sigma)$}\\
&= \pi^{\sigma}_{\downset{\tau}} \circ \zeta^{G\sigma}_{\sigma} \circ \zeta^{\sigma}_{G\sigma}(q_{\sigma}) & \\
&=\pi^{\sigma}_{\downset{\tau}}(q_{\sigma}) & \\
&=q_{\downset{\tau}}.
\end{align*}

Analogously, given $Q \in \Gamma(Y(P/G))$ one has that, for every $\downset{T} \in X(P/G)$,

\begin{align*}
\varphi(\psi(Q))_{\downset{T}} = \pi^{\Sigma}_{\downset{T}} \circ \zeta^{\sigma}_{\Sigma}(\psi(Q)_{\sigma}) &= \pi^{\Sigma}_{\downset{T}} \circ \zeta^{\sigma}_{\Sigma} \circ \pi^{\sigma}_{\sigma} \circ \zeta^{\Sigma}_{\sigma}(Q_{\Sigma})\\
&= \pi^{\Sigma}_{\downset{T}} \circ \zeta^{\sigma}_{\Sigma} \circ \zeta^{\Sigma}_{\sigma}(Q_{\Sigma}) = \pi^{\Sigma}_{\downset{T}}(Q_{\Sigma}) = Q_{\downset{T}}.
\end{align*}

\end{proof}

\section{Group actions and the Cohen-Macaulay property}
\label{sec:RACM}
\newcommand{\orb}[1]{\Sigma_{\vert #1}}

\begin{defi}\label{df:decoupled}
We call an action $H\circlearrowright \Sigma$ on a pure $d$-dimensional simplicial complex {\em decoupled} if it is translative and there is a decomposition 
$$
H=\bigoplus_{i=0}^d H_i
$$
with $H_i\neq\{0\}$ for all $i$ and such that every maximal simplex $\sigma\in \max \Sigma$ can be written as $\sigma=\{x_0,\ldots,x_d\}$ with $H_i=\stab_H(\sigma\setminus x_i)$ for all $i=0,\ldots,d$. If $\dim(\Sigma)=0$ we further require that $\stab_H(\sigma)=\{0\}$ for all $\sigma\in \Sigma\setminus \{\emptyset\}$.
\end{defi}

\begin{rem}\label{rem:fof}$\,$
\begin{itemize}
\item[(a)]
If a group action $H\circlearrowright \Sigma$ is  decoupled, then $H$ acts freely on the set of facets of $\Sigma$, i.e., $\stab_H(\beta)=\{0\}$ for every maximal face $\beta$ of $\Sigma$. This holds by definition if $\dim(\Sigma)=0$. If $\dim(\Sigma)>0$ and $h\beta =\beta$ for some $h\in H$,  then (by translativity) $h(\sigma\setminus\{x\})=(\sigma\setminus\{x\})$ for all $x\in \sigma$ and thus $h$ is in the intersection of (at least two, since $\dim(\Sigma)>0$) of the $H_i$ - which is the trivial group because the sum of the $H_i$ is direct. Thus, $h$ is the identity element.
\end{itemize}
\end{rem}

\begin{rem}\label{df:orb}
Here and in the following, given any group action $G\circlearrowright \Sigma$ on a simplicial complex and any $\sigma\in \Sigma$ we will write $$\orb{G\sigma}:=\{\tau\in \Sigma \mid \tau \subseteq g\sigma \textrm{ for some }g\in G\}$$
for the set of faces in the orbit of $\sigma$. This is a simplicial complex.
\end{rem}

\begin{lem}\label{lem_dec} Let $H\circlearrowright \Sigma$ be a decoupled action. Then, for all $\sigma\in \Sigma$ and all facets $\beta\in \max \Sigma$ with $\sigma\subseteq\beta$, the following hold.
\begin{itemize}
\item[(a)] $\stab_H(\sigma) = \bigoplus_{x\in \beta\setminus \sigma} \stab_H(\beta\setminus \{x\})$. 
\item[(b)] Let $K:=\stab_H(\sigma)$. Then $\stab_K(\beta\setminus \{x\})=\stab_H(\beta\setminus \{x\})$ for all $x\in \beta\setminus \sigma$. \\
In particular, a decoupled action on $\Sigma$ induces a decoupled action of $\stab_H(\sigma)$ on the link of $\sigma$ in $\Sigma$.
\item[(c)] Suppose that $H$ is abelian and let $L:=\bigoplus_{x\in \sigma}\stab_H(\beta\setminus \{x\})$. Then, the action of $L$ on $\orb{H\sigma}$ is isomorphic to that of the quotient $H/K$. Moreover, this action is decoupled with associated decomposition
\begin{equation}\label{dec}
L=\bigoplus_{x\in \sigma} \stab_L(\sigma\setminus \{x\}). 
\end{equation}
\end{itemize}
\end{lem}
\begin{proof}$\,$
\begin{itemize}
\item[(a)] Write $\beta=\{x_0,\ldots ,x_d\}$ and $\sigma=\{x_0,\ldots,x_d\}$, in accordance with the decomposition $H=\oplus_i H_i$. Given $h\in H$ consider its (unique) expansion $h=h_0\oplus\ldots\oplus h_d$, $h_i\in H_i$. Now, by translativity, $h\in \stab_H(\sigma)$ means $h_ix_i=x_i$ for all $i\leq k$, and this implies that $h_i$ is the identity element in $H_i$ (otherwise $h_i$ would be a nontrivial element in $H_i\cap \stab_H(\{x_i\})=\stab_H(\beta)$, contradicting Remark \ref{rem:fof}). Therefore $h\in \bigoplus_{i>k}H_i$. For the right-to-left inclusion suppose $h\in \bigoplus_{i\leq k}H_i$, i.e., the element $h_i$ is the identity in $H_i$ for all $i\leq k$.  Then, $hx_i = (h_1\ldots h_d)x_i = x_i$ for all $i\leq k$ (where we use that $h_i\in \stab_H(\{x_j\})$ for all $i\neq j$) and therefore $h\in \stab_H(\sigma)$.
\item[(b)] The first claim is immediate since
$$
\stab_K(\beta\setminus \{x\})=
\stab_H(\beta\setminus \{x\})\cap K =
\stab_H(\beta\setminus \{x\}),
$$
the first equality by definition, the second since $\stab_H(\beta\setminus \{x\})\subseteq K$ whenever $x\not\in \sigma$. On the other hand, for $x\in \sigma$ we have $K\cap \stab_H(\beta\setminus \{x\})=\{0\}$. Since $H_i\neq\{0\}$ for all $i$, we deduce that 
\begin{equation}\tag{$\dagger$}\label{rtw}
\stab_H(\beta\setminus \{x\})\subseteq K \textrm{ if and only if } x\not\in \sigma.
\end{equation}

In order to prove that the action of $K$ on the link of $\sigma$ is decoupled, we first provide a decomposition of $K$. Recall the fixed decomposition $H=\bigoplus_iH_i$. By the previous discussion we can relabel so that $H_i\cap K = H_i$ for $i\leq t$ and $H_i\cap K =\{0\}$ otherwise, and we write
$$
K=\bigoplus_{i\leq t} H_i.
$$ 
Now consider any maximal simplex $\tau$ in the link of $\sigma$ in $\Sigma$. Then, $\beta':=\tau\uplus \sigma$ is a maximal simplex of $\Sigma$. Since the action of $H$ on $\Sigma$ is decoupled, we can list $\beta'=\{x_0',\ldots,x_d'\}$ so that $H_i=\stab_H(\beta'\setminus \{x_i'\})$. Our relabeling ensures that $H_i\subseteq K$ iff $i\leq t$, and by \eqref{rtw} we also know  that $\stab_H(\beta'\setminus \{x\})\subseteq K$ iff $x\not\in \sigma$. Therefore $\tau=\{x_0',\ldots,x_t'\}$, and $H_i=\stab_H(\beta'\setminus \{x_i'\})=\stab_K(\tau\setminus \{x_i'\})$ for $i\leq t$, as required.

\item[(c)] By definition, $H=L\oplus K$. Thus $L$ is a valid choice as a set of representatives for the classes of the quotient $H/K$, with $(g\oplus K)(h\oplus K)=(gh\oplus K)$ in $H/K$. Moreover, for each $\tau\in \orb{H\sigma}$ we have $K=\stab_H(\sigma)\subseteq \stab_H(\tau)$ by translativity and because $H$ is abelian. Thus, for $g\in L$ and every $\tau\in \orb{H\sigma}$ we have that $g\tau = (g\oplus K)\tau$ and so the actions of $H/K$ and $L$ on $\orb{H\sigma}$ are equivalent under the natural isomorphism $L\to H/K$, $g\mapsto g\oplus K$. 

In order to prove that the action of $L$ is decoupled notice first that
$$
\stab_L(\sigma\setminus\{x\})=\stab_H(\sigma\setminus\{x\})\cap L = \stab_H(\beta\setminus \{x\})
$$
for all $x\in \sigma$, where the first equality is by definition and the second equality follows from part (a).
Now write $\sigma=\{x_0,\ldots,x_k\}$ and let $L_i:=\stab_L(\sigma\setminus \{x_i\})$, so that $L=\bigoplus_i L_i$ is the decomposition stated in \eqref{dec}. Every maximal simplex of the orbit of $\sigma$ is of the form $\sigma':=\{x_0',\ldots,x_k'\}$ where $x_i'=hx_i$ for all $i$ and some $h$, and all $x_i'$ are pairwise distinct. 
Now, for every $i=0,\ldots,k$, we have $\stab_L(\sigma'\setminus \{x_i'\})=\bigcap_{j\neq i}\stab_L(\{x_j'\}) = \bigcap_{j\neq i}\stab_L(\{x_j\}) = \stab_L(\sigma\setminus \{x_i\})=L_i$  as required (the second equality holds because $H$ is abelian, the first and third since the action is translative). 
\end{itemize}

\end{proof}

\begin{prop}\label{psh}
Let $H\circlearrowright \Sigma$ be a decoupled action of an abelian group $H$ and fix $\sigma\in \Sigma$. Then,  $\orb{H\sigma}$ is a shellable simplicial complex.
\end{prop}
\begin{proof}
Notice that the set $M$  of facets of $\orb{H\sigma}$ is in bijection with $L:=H/\stab_H(\sigma)$, a group whose action on $\orb{H\sigma}$ is decoupled by Lemma \ref{lem_dec}.(c). Let us consider this bijection, 
$$
\lambda: M\to L=\bigoplus_{i=1}^k L_i
$$
with $L_i=\stab_L(\sigma\setminus \{x_i\})$, where we write $\sigma=\{x_1,\ldots,x_k\}$. Notice that $L_j\subseteq \stab_L(x_i)$ for all $i\neq j$. In particular, for every $m\in M$ and all $i=1,\ldots k$ we have
\begin{equation}\tag{$\ddagger$}\label{eq:gmu}
m\cap Lx_i = \lambda(m)_i x_i.
\end{equation}

The idea now is to use $\lambda$ as a labeling of the elements of $M$. Choose a well-order $\prec_i$ of each $L_i$ that begins with the identity element $0_i$ and choose any linear extension of the cartesian product of the $\prec_i$s. Via $\lambda$ this induces a total (well-) order $\prec$ on $M$.
We will use the following property of $\prec$: if $\lambda(m)_{i} = 0_i \neq \lambda(m')_{i}$ and $\lambda(m)_j = \lambda(m')_j$ for all $j\neq i$, then $m\prec m'$. 

We want to show that $\prec$ is a shelling order for $\orb{H\sigma}$. By Remark \ref{remshell} it is enough to prove the following.

\begin{itemize}
\item[{\em Claim}] Let $m_1,m_2\in M$ with $m_1\prec m_2$ and consider $\tau\in \orb{H\sigma}$ with $\tau \subseteq m_1, \tau \subseteq m_2$. Then, there is $m_3\in M$ with $m_3\prec m_2$ and $m_1\cap m_2 \subseteq m_3\cap m_2 = m_2\setminus \{x'\}$ for some $x'\in m_2$.
\item[{\em Proof}] First notice that, since $\lambda$ is injective,  $\lambda(m_1)$ and $\lambda(m_2)$ must differ in at least one component, say the $1$st. In particular, $\lambda(m_2)_{1}\neq 0_{1}$.  

Consider $l:=(\lambda(m_2)_{1}\oplus 0_2 \oplus\cdots \oplus 0_k)\in L$ and let $m_3:=(-l) m_2$. A direct computation shows that
$$
\lambda(m_3)_{j} = 
\left\{\begin{array}{ll}
\lambda (m_2)_{j} & j\neq 1 \\
0_{1} & j=1
\end{array}\right.
$$
 We conclude:
\begin{itemize}
\item $m_3\prec m_2$ because $\lambda(m_3)_1=0_1\neq \lambda(m_2)_1$;
\item since by assumption $\lambda(m_1)_1\neq \lambda(m_2)_1$,  using Equation \eqref{eq:gmu} we obtain $m_1\cap m_2\subseteq m_3\cap m_2$;  \item choosing $x':=\lambda(m_2)_1x_1$ (i.e., the unique element of $m_2\cap Lx_1$), we have  $m_3\cap m_2= m_2\setminus \{x'\}$.
\end{itemize}
\end{itemize}

\end{proof}

\begin{thm}\label{thm:hcm}
Let $H\circlearrowright \Sigma$ be a 
decoupled action of an abelian group. If $\Sigma$ is homotopy Cohen-Macaulay, then so is $\overline{P_{\Sigma}/H}$.
\end{thm}

\begin{proof}
We proceed by induction on the dimension $d$ of $\Sigma$, the claim being trivial for $d=0$. Let then $d>0$ and suppose that the claim holds for all simplicial complexes of dimension strictly smaller than $d$. In order to prove the Cohen-Macauay property for $\overline{P_{\Sigma}/H}$ we will use the criterion given in Remark \ref{PCM_crit}, requiring us to show that open intervals in $\widehat{P_{\Sigma}/H}$ are well-connected.
Translativity of the action implies that
\begin{itemize}
\item[(i)]  $P_{\Sigma}/H$ is a simplicial poset. Thus every open interval $(x,y)\subseteq P_{\Sigma}/H$ is homotopy equivalent to a sphere of dimension $\ell(x,y)$, and is therefore well-connected.
\item[(ii)] For $H\sigma\in P_{\Sigma}/H$, there is an isomorphism ${(P_{\Sigma}/H)_{>H\sigma}}\simeq{((P_\Sigma)_{>\sigma})/\stab_H(\sigma)}$.
Now $(P_{\Sigma})_{\geq\sigma}$ is the poset of cells of the link of $\sigma$ in $\Sigma$, which is Cohen-Macaulay because $\Sigma$ is. Moreover, by Lemma \ref{lem_dec}.(b) the action of $\stab_H(\sigma)$ on the link of $\sigma$ is decoupled. If $H\sigma$ is not the minimal element of $P_{\Sigma}/H$, then the dimension of the link of $\sigma$ is less than $d$, so we can apply the induction hypothesis and conclude that $(P_{\Sigma}/H)_{> H\sigma}$ is Cohen-Macaulay (and, in particular, well-connected).
\end{itemize}
 We are now left with proving that $\overline{P_{\Sigma}/H}$ is well-connected. 
  For this, we want to apply Lemma \ref{LBWW}, in the reformulation of Remark \ref{wellconn}, to the quotient map
$
\overline{P_{\Sigma}} \mapsto\overline{P_{\Sigma}/H}
$.  
Item (ii) above verifies immediately condition (1) of Lemma \ref{LBWW}, while condition (2) holds by Proposition \ref{psh}.\end{proof}

\def\Kd{\mathscr K (\delta)}
\newcommand{\base}[1]{\underline{#1}}

\section{Semimatroids and geometric semilattices} \label{sec:matroids}

In our context, a natural analogue to matroids in classical Stanley-Reisner theory are (group actions on) semimatroids and geometric semilattices.

\subsection{Semimatroids and geometric semilattices}  

Semimatroids are abstract structures, introduced independently by Ardila and Kawahara \cite{Ard,Kawa}, that are intuitively best described as axiomatizations of the intersection pattern of a locally finite set $\mathscr{A}$ of affine hyperplanes (although the abstract theory is much more general \cite[\S 4]{DR}). Given such a set, one can single out the family $\mathcal K$ of all subsets with nonempty intersection. 
The local finiteness assumption implies that $\mathcal K$ is an abstract simplicial complex on the vertex set $\mathscr A$. Moreover, every nonempty intersection of hyperplanes is an affine subspace with a well-defined codimension: this allows us to define a function $\rho:\mathcal K \to \mathbb N$ that associates to every element of $\mathcal K$ the codimension of the corresponding intersection.   The triple $(\mathscr A, \mathcal K, \rho)$ is an example of a semimatroid.

Formally, a semimatroid is any triple $\SSS:=(S,\mathcal K, \rho)$ consisting of a set $S$, a finite-dimensional simplicial complex $\mathcal K$ on the vertex set $S$ and a function $\rho:\mathcal K \to \mathbb N$ satisfying a list of axioms that we will not need to specify (see \cite{Ard,Kawa} for the original definition and \cite{DR} for the infinite case). The {\em rank} of the semimatroid $\SSS$ is the maximum value of $\rho$, which we denote by $\rho(\SSS)$. The axioms imply that this is a finite number.

Associated to every semimatroid $\SSS$ are two posets.
\begin{itemize}
    \item The poset of {\em independent sets} is the set
    $$\mathcal I (\SSS):=\{I\in \mathcal K \mid \rho(I)=\vert I\vert\}$$
    partially ordered by inclusion. This is the poset of faces of an abstract simplicial complex on the vertex set $S$.
    \item The poset of {\em closed sets} (or {\em flats}) is the set
    $$
    \mathcal L (\SSS) :=\{ F\in \mathcal K \mid \rho(F')>\rho(F)\textrm{ for all } F'\supsetneq F\}
    $$
    partially ordered by inclusion.
\end{itemize}

Both posets are {\em geometric semilattices} in the sense of Wachs and Walker \cite{WW}. Moreover, every geometric semilattice is the poset of flats of a (possibly infinite) semimatroid \cite[Theorem E]{DR}.

We now review a notion of group actions on semimatroids. The guiding intuition here is that, in the context of our motivating example, we would like to model periodic affine hyperplane arrangements -- i.e., arrangements on which a discrete group of translations acts. In fact, when the semimatroid is associated to an affine arrangement of hyperplanes, the poset of flats is isomorphic to the poset of all intersections of subspaces in the arrangement ordered by reverse inclusion (see, e.g., Section \ref{sec:AA}).

\begin{defi}[Compare {\cite[\S 3]{DR}}]\label{def_gaos}
Let $G$ be a group. A {\em $G$-semimatroid} $\mathfrak S : G\circlearrowright \SSS$ is an action of $G$ on a semimatroid $\SSS=(S,\mathcal K,\rho)$, i.e., an action of $G$ by permutations of $S$ that preserves $\mathcal K$ and $\rho$.  
Furthermore, we require that there is a finite number of orbits of elements of $\mathcal K$. The rank of $\mathfrak S$ is $\rho(\mathfrak S):=\rho(\SSS)$, the maximum of $\rho$ over $\mathcal K$.

Such a $G$-semimatroid is called {\em translative} if, for every $s\in S$, $\{g(s),s\}\in \mathcal K$ implies $g(s)=s$.
The $G$-semimatroid is called {\em ($k$-)refined} if, in addition to being translative, $G$ is a free Abelian group and there is $k\in \mathbb N$ such that, for every $x\in \mathcal K$, $\stab(x)$ is a direct summand of rank $k(\rho(\SSS)-\rho(x))$.
 \end{defi}

 Every $G$-action on a semimatroid $\SSS$ induces an action of $G$ by automorphisms on the posets $\mathcal I(\SSS)$ and $\mathcal L (\SSS)$.  If the $G$-semimatroid is translative, resp.\ refined, then so are the induced actions on both posets.
 Conversely, every action on $\mathcal I (\SSS)$ induces an action on $\SSS$ and every action on a geometric semilattice induces an action on the associated (simple) semimatroid.

 \begin{defi} Given a $G$-semimatroid $\mathfrak S:G\circlearrowright \SSS$  define  the posets
 $$
   \mathcal K_{\mathfrak S} := \mathcal K (\SSS) / G
 \quad\quad\quad
 \mathcal I_{\mathfrak S} := \mathcal I (\SSS) / G
 \quad\quad\quad
  \mathcal P_{\mathfrak S} := \mathcal L (\SSS) / G.
 $$
 \end{defi}
 
 \begin{rem} $\,$\label{rem:qps}
 \begin{itemize}
     \item[(i)] By the finiteness requirement in Definition \ref{def_gaos}, $\IS$, $\PS$ and $\mathcal K_{\mathfrak S}$ are all finite.
 
\item[(ii)] Both $\IS$ and $\PS$ are bounded-below and ranked of length $\rho(\SSS)$. The poset rank of an element $X$ of either poset equals $\rho(X)$.
  
\item[(iii)] If $\mathfrak S$ is a translative $G$-semimatroid, then $\mathcal I_{\mathfrak S}$ and $\mathcal K_{\mathfrak S}$ are simplicial posets.
  \end{itemize}
 \end{rem}

The above definitions and terminology were motivated in \cite{DR} by the case of periodic affine hyperplane arrangements related to Abelian arrangements, as we will discuss later. However, these definitions are strictly more general, see \cite{DR}.

\begin{defi} Let $E_{\mathfrak S}:=S/G$ denote the (finite) set of orbits of elements of $\SSS$. The {\em support} of an orbit is given by the function 
$$
\Supp: \mathcal K_{\mathfrak S} \to 2^{E_{\mathfrak S}}, \quad\quad
G\{x_1,\ldots,x_k\} \mapsto \{Gx_1,\ldots,Gx_k\}.
$$
For every $\alpha\in \mathcal K$, we will write $\supp(\alpha):=\Supp(G\alpha)$, the support of the orbit of $\alpha$.

In general, given $A\subseteq E_{\mathfrak S}$ we let $m_{\mathfrak S}(A):=\vert\Supp^{-1}(A) \vert$ be the number of all orbits $X \in \mathcal K_{\mathfrak S}$ such that $\Supp{X}=A$ - thus $m(A)\neq 0$ if and only if $A\in \Supp{\mathcal K_{\mathfrak S}}$.

Moreover, for $A\subseteq E_{\mathfrak S}$ write $\underline{\rho}(A)$ for the rank $\rho(X)$ of any $X\in \Supp^{-1}(A)$ and let $\underline{\SSS}:=(E_{\mathfrak S}, \Supp(\mathcal K_{\mathfrak S}),\underline{\rho})$.  
\end{defi}

\begin{rem}\label{rem:unse}
The set $\Supp(\mathcal I_{\mathfrak S})$ is a simplicial complex. The triple $\underline{\SSS}$ is a semimatroid 
 if and only if $\mathfrak S:G\circlearrowright \SSS$ is translative \cite[Theorem A]{DR}. In this case, $\Supp(\mathcal I_{\mathfrak S})$ is the associated simplicial complex of independent sets.
\end{rem}

\subsection{Tutte polynomials and $h$-polynomials}

\begin{defi} Let $\mathfrak S: G\circlearrowright \SSS$ denote a translative action of $G$ on a semimatroid $\SSS$ of rank $d$. The Tutte polynomial of $\mathfrak S$ is
$$
T_{\mathfrak S}(x,y):= \sum_{A\subseteq E_{\mathfrak S}}
m_{\mathfrak S}(A)(x-1)^{d-\underline{\rho}(A)}(y-1)^{\vert A \vert - \underline{\rho}(A)}.
$$
\end{defi}

\begin{rem}
Notice that $A\subseteq E_{\mathfrak S}$ is  central in the semimatroid $\underline{\SSS}$ if and only if $m_{\mathfrak S}(A)\neq 0$, hence $T_{\mathfrak S}(x,y)$ is a "weighting" of the Tutte polynomial of $\underline{\SSS}$ defined in \cite{Ard}. If $\mathfrak S$ is associated to the action of the group of translations of a periodic hyperplane arrangement and $\underline{\SSS}$ is in fact a matroid, then $T_{\mathfrak S}(x,y)$ is the arithmetic Tutte polynomial of the corresponding toric arrangement \cite{Moc1}.
\end{rem}

\begin{lem}\label{lem:pols}
Let $\mathfrak S$ be a translative $G$-semimatroid of rank $d$. Then
\begin{itemize}
    \item[(i)] 
The $h$-polynomial of the simplicial poset $\IS$ is $$h_{\IS}(t)=t^dT_{\mathfrak S}({1 / t},1),$$
and the characteristic polynomial of $\IS$ is
$\chi_{\IS}(t) = (-t)^dT_{\mathfrak S}(\frac{1-t}{t},1)$.
\item[(ii)] The characteristic polynomial of the poset $\mathcal P_{\mathfrak S}$ is
$$
\chi_{\mathcal P_{\mathfrak S}}(t)  = (-1)^d T_{\mathfrak S}({1-t},0).
$$
\end{itemize}
\end{lem}
\begin{proof} Item (ii) is \cite[Theorem F]{DR}. 
In order to prove item (i), start by noticing that the number of orbits of independent sets of rank $i$ in $\II_{\mathfrak S}$ is
$$
f_{i-1}(\IS)=\sum_{\substack{A\subseteq E_{\mathfrak S} \\ \underline{\rho}(A)=|A|=i}} m_{\mathfrak S}(A).
$$
Therefore, with Remark \ref{rem:qps}.(ii) we can write
$$
h_{\IS}(t)\stackrel{\textrm{df.}}{=}t^d\sum_{i=1}^d f_{i-1}(P) \bigg(\frac{1-t}{t}\bigg)^{d-i}=t^d \sum_{\substack{A\subseteq E_{\mathfrak S} \\ \underline{\rho}(A)=\vert A \vert}} m_{\mathfrak S}(A)\bigg(\frac{1-t}{t}\bigg)^{d-\underline{\rho}(A)}
$$
$$
=t^d\sum_{\substack{A\subseteq E_{\mathfrak S} }} m_{\mathfrak S}(A)\bigg(\frac{1}{t}-1\bigg)^{d-\underline{\rho}(A)}0^{\vert A \vert - \underline{\rho}(A)}\stackrel{\textrm{df.}}{=} t^d T_{\mathfrak S}\bigg(\frac{1}{t},1\bigg).
$$
The formula for the characteristic polynomial of $\IS$ follows with Remark \ref{rem:chih}.
\end{proof}

\subsection{Some structure theory of $G$-semimatroids}
Throughout this section let $\mathfrak S:G\circlearrowright \SSS$ denote the action of a group $G$ on a semimatroid $\SSS=(S,\mathcal K, \rho)$. First, let us recall the notion of contraction and deletion for semimatroids. In the following, given any collection $\mathcal X\subseteq 2^S$ of subsets of a set $S$ and given any $A\subseteq S$, we write $\mathcal X/A:=\{X\subseteq S\setminus A \mid X\cup A\in \mathcal X\}$ and $\mathcal X[A]:=\{X\in \mathcal X \mid X\subseteq A\}$ (the latter system of sets is commonly also written as $\mathcal X \setminus (S\setminus A)$).

\begin{defi}[see, e.g., {\cite[\S{} 7]{Ard}}] Let $\SSS=(S,\mathcal K, \rho)$ be a semimatroid and let $\alpha\subseteq S$. The  {\em restriction} of $\SSS$ to $\alpha$ is the semimatroid 
$$
\SSS[\alpha]:=(\alpha, \mathcal K[\alpha], \rho_{\vert \mathcal K[\alpha]})
$$
If $\alpha\in \mathcal K$, we can also define the {\em contraction} of $\alpha$ in $\SSS$ as
$$
\SSS/\alpha :=(S\setminus \alpha, \mathcal K / \alpha, \rho_{/\alpha})
$$
where $\rho_{/\alpha}(X):=\rho(X\cup \alpha)-\rho(\alpha)$ for all $X\in \mathcal K/\alpha$.
\end{defi}

\begin{rem}\todo{Controlla}
For every $A\subseteq \Supp(\mathcal K_{\mathfrak S})$ we can choose $\alpha\in \mathcal K$ such that $\supp{\alpha}=A$. If $G$ is abelian and  $\mathfrak S$ is translative, then $\stab(\alpha)=\bigcap_{s\in \alpha}\stab(s)$ does not depend on the choice of $\alpha$. In particular, $\stab(A)$ is well-defined as the stabiliser of any such $\alpha$.
\end{rem}

\begin{defi}
Let $\mathfrak S: G\circlearrowright \SSS$ be a refined $G$-semimatroid and let $X\subseteq E_{\mathfrak S}$. As is customary, write $\cup X:=\bigcup_{x\in X} x$. The {\em restriction} of $\mathfrak S$ to $X$ is
$$
\mathfrak S[X]: G/\stab(X)\circlearrowright \SSS[\cup X].
$$

Moreover, for any $A \in {\mathcal K_{\mathfrak S}}$ we define the {\em contraction}
$$
\mathfrak S/A : \stab(\alpha)\circlearrowright \SSS/\alpha
$$
where $\alpha$ is any representative of the orbit $A$, i.e., $A=G\alpha$.
\end{defi} 

\todo{ESEMPIO!!}

\begin{rem}\label{rem:minorsup}
For every $\alpha\in\mathcal K$ there are poset isomorphisms 
$$\mathcal L (\SSS/\alpha) \simeq 
\mathcal L (\SSS)_{\geq \operatorname{cl}(\alpha)}, 
\quad\quad \mathcal I (\SSS/\alpha) \simeq  \mathcal I (\SSS)_{\geq \beta}
$$ where  $\beta$ is any maximal independent subset of $\alpha$ and $\operatorname{cl}$ denotes semimatroid closure (see \cite[Definition 3.27]{DR}). Thus,
$$\mathcal P_{\mathfrak S / G\alpha} 
\simeq (\mathcal P_{\mathfrak S} )_{\geq G\alpha}
$$ for every $G\alpha\in \mathcal P_{\mathfrak S}$, and
$$\mathcal I_{\mathfrak S / G\alpha} 
\simeq (\mathcal I_{\mathfrak S} )_{\geq G\beta}$$
for all $\alpha\in \mathcal K$ and every maximal independent $\beta\subseteq \alpha$.
\end{rem}

\begin{rem}
From Lemma \ref{lem:up} and Lemma \ref{lem:sg} follows that restrictions and contractions of refined $G$-semimatroids are refined. 
\end{rem}

\subsection{Refined quotients of independence complexes}

In this section let $\mathfrak S: G\circlearrowright \SSS$ be a refined action.

\begin{defi}
Given any $X=\{x_1,\ldots,x_k\}\in \Supp({\mathcal I_{\mathfrak S}})$ let
$$
G^{(X)}:=G/\stab_G(X), \quad\quad H^{(X)}:=\bigoplus_{i=1}^k H^{(X)}_i \quad
\textrm{ where }H^{(X)}_i:=\stab_{G^{(X)}}(X\setminus \{x_i\}).
$$ 
Then $H^{(X)}$ is free abelian and of maximal rank in $G^{(X)}$, thus we can define the number
$$
\delta_{\mathfrak S}(X):=[G^{(X)}:H^{(X)}].
$$
Moreover, let
$$
\delta_{\mathfrak S} := \operatorname{lcm} \{\delta_{\mathfrak S}(X) \mid X\in \Supp{\mathcal I_{\mathfrak S}}\}.
$$
\end{defi}

This definition is motivated by the following fact.

\begin{lem} \label{lem_scm}
Let $X\in \Supp({\overline{\mathcal I_{\mathfrak S}}})$. Then $\overline{\mathcal I_{\mathfrak S [X]}}$ is well-connected in characteristic $0$ and every characteristic not dividing $\delta_{\mathfrak S}(X)$.
\end{lem}
\begin{proof}
Recall that $\mathfrak S [X]$ is  $G^{(X)}\circlearrowright\SSS[\cup X]$, and induces a refined action $G^{(X)}\circlearrowright\mathfrak I [\cup X]$. The induced action $H^{(X)}\circlearrowright \mathcal I[\cup X]$ is decoupled with respect to the decomposition $H^{(X)}=\bigoplus_{i=1}^{k}H_i^{(X)}$. In order to see this notice first that, since $\mathfrak S [X]$ is refined, the rank of the (free abelian) group $H_i^{(X)}$ is a nonzero multiple of $\underline{\rho}(X)=\vert X \vert$, and since $X$ is not empty by assumption we conclude that $H_i^{(X)}$ is not the trivial group. The fact that the action of $G^{(X)}$ is refined also implies that the stabilizer of any maximal face of $\mathcal I[\cup X]$ is $\stab_{H^{(X)}}(X)\subseteq\stab_{G^{(X)}}(X)=\{0\}$, thus the action of $H^{(X)}$ is free on facets. 
Now write $X=\{x_1,\ldots,x_k\}$. Given $\beta=\{s_1,\ldots,s_k\}$ with $s_i\in x_i$ for all $i$, 
we have $
\stab_{H^{(X)}}(\beta\setminus \{s_i\})=
\stab_{G^{(X)}}(X\setminus x_i)=H_i^{(X)}
$, where the first equality holds because freeness on facets implies that in the definition of $H^{(X)}$ every summand other than $\stab_{G^{(X)}}(X\setminus x_i)$ fails to stabilize some element of $\beta\setminus \{s_i\}$. 

Thus,  $\overline{\mathcal I[\cup X]/H^{(X)}}$ is homotopy Cohen-Macaulay by Theorem \ref{thm:hcm}. Moreover, the action of the (finite) group $K:=G^{(X)}/H^{(X)}$ on $\mathcal I[\cup X]/H^{(X)}$ is translative (Lemma \ref{lem:quac}.(b)). Now by Lemma \ref{lem:quac}.(a) we have $\mathcal I_{\mathfrak S[X]}\simeq(\mathcal I[\cup X]/H^{(X)})/K$, and the claim follows with Lemma \ref{lem_bredon} and \cite[Chapter III, Theorem 2.4]{Bredon}. 
\end{proof}

We now proceed with two properties of the numbers $\delta_{\mathfrak S}$.

\begin{lem}\label{lem:sempli} If $X\leq Y$ in $\Supp(\mathcal I_{\mathfrak S})$, then $\delta_{\mathfrak S}(X)$ divides $\delta_{\mathfrak S}(Y)$. In particular
$$
\delta_{\mathfrak S} = \operatorname{lcm}\{\delta_{\mathfrak S}(B) \mid B\in \max \Supp(\mathcal I_{\mathfrak S})\},
$$
i.e., $\delta_{\mathfrak S}$ can be computed as the least common multiple of the $\delta_{\mathfrak S}(B)$ where $B$ ranges over the bases of the semimatroid $\underline{\mathscr S}$.
\end{lem}
\begin{proof} Translativity of the action implies $\stab_G(Y)\subseteq \stab_G(X)$ and thus the existence of a surjective (quotient) homomorphism $\pi:G^{(Y)}\to G^{(X)}$. Now, for every $y\in Y\setminus X$, $\stab_{G}(Y\setminus \{y\})\subseteq \stab_G (X)$ and therefore $\pi(\stab_{G^{(Y)}}(Y\setminus \{y\}))=\{0\}$. On the other hand, if $x\in X$ then $\stab_{G}(Y\setminus \{x\})\subseteq \stab_G (X\setminus \{x\})$ and, passing to the quotients, $\pi(\stab_{G^{(Y)}}(Y\setminus \{x\}))\subseteq \stab_{G^{(X)}} (X\setminus \{x\})$. Thus, $\pi$ restricts to a map $H^{(Y)}\to H^{(X)}$, and so it induces a group homomorphism $q: G^{(Y)}/H^{(Y)} \to G^{(X)}/H^{(X)}$ fitting in the following diagram where we see that (e.g. by the Snake Lemma) the cokernel of $q$ is trivial.
\begin{center}
\begin{tikzcd}
0\arrow[r]&
H^{(Y)}\arrow[r,hook]\arrow[d,dashed]
& G^{(Y)} \arrow[r]\arrow[d,twoheadrightarrow,"\pi"]
& G^{(Y)}/H^{(Y)} \arrow[r]\arrow[d,"q"] & 0 \\
0\arrow[r]&H^{(X)}\arrow[r,hook]
& G^{(X)} \arrow[r]\arrow[d]
& G^{(X)}/H^{(X)} \arrow[r]\arrow[d]&0 \\
& & 0\arrow[r] & \operatorname{coker}(q) \arrow[r] & 0 
\end{tikzcd}
\end{center}
Surjectivity of $q$ implies that $\delta_{\mathfrak S}(X)=\vert G^{(X)}/H^{(X)}\vert$ divides $\delta_{\mathfrak S}(Y)=\vert G^{(Y)}/H^{(Y)}\vert$.
\end{proof}

\begin{lem}\label{lem_divide}
Let $A\in \mathcal K_{\mathfrak S}$. Then 
$
\delta_{\mathfrak S/A}
\textrm{ divides } \delta_{\mathfrak S}
$. 
In particular, $CM(\delta_{\mathfrak S/A})$ implies $CM(\delta_{\mathfrak S})$.
\end{lem}
\begin{proof}
Choose $\alpha\in \mathcal K$ such that $A=G\alpha$ and a $\beta=\{s_1,\ldots,s_k\}\in \mathcal I$ maximal such that $\beta\subseteq \alpha$. Then, let $K:=\stab_G(\alpha)=\stab_G(\beta)$ (the equality because of \cite[Lemma 8.1.(b)]{DR}, where only translativity is used) and notice that $\mathcal I_{\mathfrak S / A}\simeq \mathcal I_{\geq \beta}/K$ (by Remark \ref{rem:minorsup} and Lemma \ref{lem:up}
.(ii)). 

Given $X\in \Supp(\mathcal I_{\mathfrak S/A})$, choose $\sigma=\{s_{k+1},\ldots, s_l\}$ such that $X=\Supp(K\sigma)$ and $\beta\uplus \sigma\in \mathcal I$, and let $$X':=\Supp(G(\beta\uplus\sigma))\in \Supp(\mathcal I_{\mathfrak S}).$$
It is now enough to prove that, for every such $X$, $\delta_{\mathfrak S/A}(X)$ divides $\delta_{\mathfrak S}(X')$. 
In order to do that, first notice that with the notation introduced above we have
\begin{align*}
K/\stab_K(X) 
&\stackrel{\operatorname{df}}{=}\stab_G(\beta)/\stab_{\stab_G(\beta)}(\sigma)
\stackrel{\operatorname{df}}{=}\stab_G(\beta)/\stab_G(\beta)\cap\stab_G(\sigma)\\
&=\stab_G(\beta)/\stab_G(\beta\cup\sigma)
\stackrel{\operatorname{df}}{=}\stab_{G^{(X')}}(\beta)
\end{align*} 
where the equality at the break of the line uses translativity of the action (Remark \ref{rem_traeq}).
Using this identity and noticing that, for every $i=k+1,\ldots,l$ we have  $$\stab_{\stab_{G^{(X')}}(\beta)}(\sigma\setminus\{s_i\}) = \stab_{G^{(X')}}(\beta\uplus \sigma\setminus\{s_i\}),$$ we can expand the definitions.
We start by writing
$$
\delta_{\mathfrak S/A}(X) = 
\left[
\stab_{G^{(X')}}(\beta):
\bigoplus_{i=k+1}^l\stab_{G^{(X')}}(\beta\uplus \sigma\setminus\{s_i\})
\right].
$$
Let us call $U$ the direct sum on the right-hand side. We compare this with
$$
\delta_{\mathfrak S}(X') = 
\left[
G^{(X')}:
\bigoplus_{i=1}^l\stab_{G^{(X')}}((\beta\uplus \sigma)\setminus\{s_i\})
\right].
$$
Now since the action of $G^{(X')}$ is refined, we can write $G^{(X')}=\stab_{G^{(X')}}(\beta) \oplus W$ for some subgroup $W$. Setting for brevity 
$U':= \bigoplus_{i=1}^{k}\stab_{G^{(X')}}(\beta\uplus\sigma\setminus\{s_i\})$, it is now enough to prove that the map
$$
f: 
\frac{\stab_{G^{(X')}}(\beta)}{U} \to 
\frac{\stab_{G^{(X')}}(\beta)\oplus W}{U\oplus U'}, \quad g+ U \mapsto g\oplus 0 + U\oplus U' 
$$
is an injection. In fact, in this case $\delta_{\mathfrak S/A} (X)= \vert \stab_{G^{(X')}}(\beta)/U \vert$ is the cardinality of a subgroup of a group of cardinality $\delta_{\mathfrak S}(X')$, thus the former divides the latter.

Now injectivity of $f$ is a straightforward computation, thus the lemma is proved.
\end{proof}

\begin{thm}\label{thm_CMSM}
Let $\mathfrak S$ be a refined $G$-semimatroid. Then $\overline{\mathcal I_{\mathfrak S}}$ is $CM(\delta_{\mathfrak S})$ (i.e., Cohen-Macaulay in characteristic $0$ and every characteristic that does not divide $\delta_{\mathfrak S}$).
\end{thm}

\begin{proof}[Proof of Theorem \ref{thm_CMSM}]
We proceed by induction on the rank of $\mathfrak S$, the case of rank $0$ being trivial. Suppose then that $\mathfrak S$ has positive rank and notice that,  for every $X\in \overline{\mathcal I_{\mathfrak S}}$, the poset ${\mathcal I_{\mathfrak S}}_{> X}\simeq \overline{\mathcal I_{\mathfrak S/X}}$ is $CM(\delta_{\mathfrak S/\Supp(X)})$ by induction hypothesis, and thus also $CM(\delta_{\mathfrak S})$ by Lemma \ref{lem_divide}. Since every lower open interval in $\overline{\mathcal I_{\mathfrak S}}$ is the boundary of a simplex (because $\mathcal I_{\mathfrak S}$ is a simplicial poset), in order to prove the claim it is enough to prove that $\overline{\mathcal I_{\mathfrak S}}$ is well-connected in characteristic $0$ and every characteristic that does not divide $\delta_{\mathfrak S}$.

To this end, consider the restriction of the (order-preserving) support map
$$
\Supp :\overline{\mathcal I_{\mathfrak S}}\to \overline{\Supp({\mathcal I_{\mathfrak S}})}.
$$
We know that, since the action is translative, the poset $\Supp({\mathcal I_{\mathfrak S}})$ is the geometric semilattice of independent sets of a semimatroid - and thus it is homotopy Cohen-Macaulay \cite{BiPr}. In particular, $\Supp({\mathcal I_{\mathfrak S}})_{>X}$ is well-connected for all $X\in \mathcal I_{\mathfrak S}$. 
Now, one also checks that, for every $X\in \overline{\mathcal I_{\mathfrak S}}$, $$\Supp^{-1}(\Supp({({\overline{\mathcal I_{\mathfrak S}}})}_{\leq X}))=\overline{\mathcal I_{\mathfrak S[X]}}.$$ 
By Lemma \ref{lem_scm}, this poset is is well-connected in characteristic $0$ and every characteristic not dividing $\delta_{\mathfrak S}(X)$. 
 Now an application of Lemma \ref{LBWW} proves the claim.
\end{proof}

\subsection{On Stanley-Reisner rings of $G$-semimatroids}

In analogy with (and extending) classical matroid theory, it is now natural to state the following definition.
\begin{defi}
Given a $G$-semimatroid $\mathfrak S$ let 
$$
\mathcal R_{\mathfrak S} := \mathcal R( \mathcal I_{\mathfrak S})
$$
be the Stanley-Reisner ring of $\mathfrak S$.
\end{defi}

\begin{rem}
There is another class of matroidal simplicial complexes whose Stanley-Reisner rings have been in the focus of a substantial amount of work, namely the so-called ``no-broken-circuit complexes'' \cite{Bryl}. No-broken-circuit sets for $G$-semimatroids (and their Stanley-Reisner rings) will be treated in forthcoming work.  
\end{rem}

From our results the following facts follow immediately.

\begin{prop} Let $\mathfrak S$ be a $G$-semimatroid of rank $d$.
\begin{itemize}
\item If $\mathfrak S$ is translative, $\mathcal R_{\mathfrak S}$ is isomorphic to the Stanley ring associated to the (finite) simplicial poset $\mathcal I_{\mathfrak S}$.
    \item If $G$ is the trivial group, $R_{\mathfrak S}$ is isomorphic to the classical Stanley-Reisner ring of the (independence complex of the) underlying (semi)matroid.
\item If $\mathfrak S$ is refined, then the poset $\IS$ is $(d-2)$-connected, and $$\widetilde{H}_{d-1}(\IS,\mathbb Q)\simeq \mathbb Q^{-T_{\mathfrak S}(0,1)}.$$
\item If $\mathfrak S$ is refined, then the associated Stanley-Reisner ring is Cohen-Macaulay, with $h$-polynomial $h_{\IS}(t)=t^d T_{\mathfrak S}(1/t,1)$.
\end{itemize}
\end{prop}

\subsection{On refined quotients of geometric semilattices}

As a byproduct of our previous considerations we can prove the following result on the topology of quotients of geometric semilattices.

\begin{thm}\label{thm:PGCM}
If $\mathfrak S$ is a refined $G$-semimatroid, then the poset  $\check{\mathcal P_{\mathfrak S}}$ is $CM(\delta_{\mathfrak S})$.
\end{thm}

We postpone the proof of this theorem until after some preparatory work.

\begin{rem}
The restriction on the characteristic in Theorem \ref{thm:PGCM} is only due to the corresponding limitation in Theorem \ref{thm_CMSM}. In fact, Proposition \ref{prop:PLCM} shows that in any contraction-closed class of group actions on semimatroids, $\overline{\mathcal P_{\mathfrak S}}$ is Cohen-Macaulay "of the same class" as $\overline{\mathcal I_{\mathfrak S}}$. 
\end{rem}

\begin{cor}\label{cor:HT}
More precisely, if $\mathfrak S : G\circlearrowright \SSS$ is a refined action on a semimatroid of rank $d$,
$$\widetilde{H}_i(\check{\PS},\mathbb Q)=
\left\{\begin{array}{ll}
\{0\} & \textrm{if }i\leq d-2 \\
\mathbb Q^{T_{\mathfrak S}(0,0)}
& \textrm{if }i=d-1 
\end{array}\right.$$
If $\mathcal P/G$ is bounded above, then clearly $\mathcal P/G\setminus \{\hat 0\}$  is contractible. In this case,
$$\widetilde{H}_i({\sz{\PS}},\mathbb Q)=
\left\{\begin{array}{ll}
\{0\} & \textrm{if }i\leq d-2 \\
\mathbb Q^{-T_{\mathfrak S}(1,0)}
& \textrm{if }i=d-1 
\end{array}\right.$$

\end{cor}
\begin{proof}
The Corollary's claim for $i\leq d-2$ is a reformulation of the connectivity claim in the Theorem. The claims for $i=d-1$ follow using Lemma \ref{lem:EuCar} and Lemma \ref{lem:pols}.
\end{proof}

The following proposition is the key tool in the proof of Theorem \ref{thm:PGCM}.

\begin{prop}\label{prop:PLCM} If $\check{\IS}$ as well as every $\check{\mathcal P_{\mathfrak S / p}}$ for all $p\in \check{\PS}$  are well-connected, then $\check{\PS}$ is homotopy Cohen-Macaulay. The homological version of the claim also holds (by fixing a characteristic, say $k$, and replacing ``well-connected'' with ``acyclic in characteristic $k$ through codimension $1$'' and ``homotopy Cohen-Macaulay'' with ``Cohen-Macaulay in characteristic $k$'').
\end{prop}

\begin{proof} 
By Remark \ref{PCM_crit} we have to prove that every open interval $(x,y)\subseteq \widehat{\PS}$  is $(\ell(x,y) -1)$-connected (resp., if a characteristic $k$ is fixed, $(\ell(x,y) -1)$-acyclic in characteristic $k$). If $y\neq \hat 1$ this is true because Lemma \ref{lem:low} implies that bounded intervals in $\PS$ are isomorphic to bounded intervals in $\mathcal L (\SSS )$, but bounded intervals in $\mathcal L (\SSS )$ are geometric lattices 
\cite[Theorem 2.1]{WW}, hence their reduced order complexes are (homotopically) well-connected (and in particular acyclic through codimension $1$ in every characteristic). We are left with proving that, for every $Gx\in \PS$, the poset $(\PS)_{>Gx}$ is $(\ell((\PS)_{>Gx}) -1)=(\rk((\PS)_{\geq Gx})-2)$-connected (resp.\ $(\rk((\PS)_{\geq Gx})-2)$-acyclic  in characteristic $k$). If $Gx\in \check{\PS}$ this is true by assumption since, with Lemma \ref{lem:up}, we have that $(\PS)_{\geq Gx}=\mathcal P_{\mathfrak S / Gx}$.

We are left with proving that $\check{\PS}$ is $(\rho(\mathfrak S)-2)$-connected (resp.\ -acyclic). This will follow from the assumption on $\IS$ via Lemma \ref{LBWW}.
The posets $\check{\IS}$ and $\check{\PS}$ are both ranked of the same length $(\rho(\mathfrak S)-1)$.  The equivariant and rank-preserving poset map 
$\operatorname{cl}: \mathcal I (\SSS) \to \mathcal L (\SSS)$ given by semimatroid closure (see \cite[Definition 3.27]{DR}) 
induces a rank preserving poset map 
$$f: \check{\IS}\to \check{\PS}, \quad\quad GI \mapsto G\operatorname{cl}(I)$$

The claim follows by Lemma \ref{LBWW} applied to $f$ with $t=(\rho(\mathfrak S)-2)$. Thus we only have to check that Lemma's assumptions. Let henceforth $\rk$ denote the rank function of the poset $\PS$ and fix $p\in \check{\PS}$. Then, $\operatorname{rk}(p)>0$ and $\ell((\check{\PS})_{<p})= \rk(p)-2$ (where we take the length of the empty poset to be $-1$).
\begin{itemize}
\item[(1)] By Lemma \ref{lem:sg}.(i) and Lemma \ref{lem:up}.(ii), the poset $(\check{\PS})_{>p}$ is isomorphic to $\check{\mathcal P_{\mathfrak S /p}}$ and thus by assumption it is $(d-\operatorname{rk}(p)-2)=
(t - \ell((\check{\mathcal P_{\mathfrak G}})_{< p}) - 2)$-connected (resp.\ acyclic over $\mathbb K$).
\item[(2)] We are left with showing that $f^{-1}((\check{\PS})_{\leq p})$ is  $(\rk(p)-2)$-connected. This will follow from the fact that it is isomorphic to the poset of (proper) faces of the independence complex of a rank $\rk(p)$ matroid, which is classically known to be $(\rk(p)-2)$-connected \cite{BjoAltro}, and thus in particular also $(\rk(p)-2)$-acyclic in every characteristic. This isomorphism is proved in the next claim which, then, concludes the proof of the theorem.

More precisely, we fix a representative $F\in p$  
 and consider the (rank $\rk(p)$) matroid $\SSS[F]$, the restriction of $\SSS$ to $F$  
 We write $\mathcal I [F]$ for the poset of independent sets of this matroid and note that the poset of flats of $\SSS[F]$ is naturally isomorphic to $\mathcal L (\SSS)_{\leq F}$. 

\begin{itemize}
\item[] {\bf Claim.} 
We claim that the quotient map by the $G$-action induces a poset isomorphism
$$
\gamma: \check{\mathcal I}[F] \to f^{-1}((\check{\PS})_{\leq p}). 
$$
\item[] {\em Proof of claim.} Since both posets are finite, it will suffice to prove that $\gamma$ is a bijective order-preserving map. Write $\mathcal L$ instead of $\mathcal L(\SSS)$ for brevity, and consider the diagram

\begin{center}
\begin{tikzpicture}[x=5em,y=2em]
\node at (-1,1) (A) {$f^{-1}((\check{\PS})_{\leq p})$};
\node at (1,1) (B) {$(\check{\PS})_{\leq p}$};
\node at (1,-1) (C) {$\check{\mathcal L}_{\leq F}$};
\node at (-1,-1) (D) {$\check{\mathcal I}[F]$};
\draw[->] (A) -- (B);
\node at (0,1.3) (cld) {$f$};
\draw[->] (D) -- (A);
\node at (-.8,0) (gm) {$\gamma$};
\draw[->] (D) -- (C);
\node at (0,-.7) (cl) {$\operatorname{cl}$};
\draw[->] (C) -- (B);
\node at (0.8,0) (qt) {$\alpha$};
\end{tikzpicture}
\end{center}
where $\operatorname{cl}$ is the closure map of the matroid $\SSS[F]$ and $\alpha: X\mapsto GX$ denotes the restriction of the quotient map of the action on $\mathcal L$. 
The maps $f$, $\operatorname{cl}$ and $\alpha$ are rank-preserving by definition, and $\alpha$ is a poset isomorphism because the group action is translative.

For every $I\in \check{\mathcal I}[F]$, since $I\subseteq F$ and $\operatorname{cl}(F)=F$ we have $\operatorname{cl}(I) \in \check{\mathcal L}_{\leq F}$. 
Unwrapping the definitions we see that  $f\gamma(I) = f\operatorname{cl}(I) = \alpha \operatorname{cl} (I) \in (\check{\PS})_{\leq p}$, thus the map $\gamma$ is well-defined and the diagram commutes. That $\gamma$ is order-preserving follows because it is the restriction of the (order-preserving) quotient map on $\mathcal I (\SSS)$.

Moreover, given any $q\in f^{-1}((\check{\PS})_{\leq p})$ consider the element $X:=\alpha^{-1}(f(q))$ and let $I\subseteq X$ be such that $q=GI$. Then $I$ is independent and $I\subseteq X \subseteq F$, hence $I\in\check{\mathcal I}[F]$ and clearly $\gamma(I)=GI=q$, hence $\gamma$ is surjective.

Finally, any $I'\in \check{\mathcal I}[F]$ with $\gamma I'=q=GI$ satisfies $gI'=I$ for some $g\in I$ hence, by translativity of the action on ${\mathcal I (\SSS)}$, we must have $I=I'$ and so $\gamma$ is injective.

As a bijective order-preserving map between finite posets, $\gamma$ is a poset-isomorphism as claimed.
\end{itemize}
\end{itemize}
\end{proof}

\begin{proof}[Proof of Theorem \ref{thm:PGCM}]  
We argue by induction on $d:=\rho(\mathfrak S)$. The claim trivially holds if $\rho(\mathfrak S)=0$. Now suppose that $\mathfrak S$ has rank $d>1$ and that the claim holds in every lower rank. By Theorem \ref{thm_CMSM}, $\check{\IS}$ is $CM(\delta_{\mathfrak S})$. Moreover, for every $p\in \check{\mathcal P_{\mathfrak S}}$ by induction hypothesis the poset $\check{\mathcal P_{\mathfrak S/p}}$ is $CM(\delta_{\mathfrak S /p})$, hence in particular $CM(\delta_{\mathfrak S})$ (Lemma \ref{lem_divide}). Thus we conclude by applying Proposition \ref{prop:PLCM}. 

\end{proof}

\section{
Applications to arrangements}\label{sec:AA}

In this section we show that our definitions, and the level of generality of our theorems, do encompass one of the main motivating examples, namely that of certain algebraically defined arrangements of submanifolds which generalize the classical setting of arrangements of hyperplanes in vector spaces. Just as every hyperplane arrangement has an associated (semi)matroid, these larger classes of geometric objects have a natural associated $G$-semimatroid (the case of hyperplanes being recovered by trivial group actions). Here we review the definition of Abelian (incl.\ toric and elliptic) arrangements and of $(p,q)$-arrangements, and we prove that the associated Stanley-Reisner rings satisfy our theorems. Even though $(p,q)$-arrangements can be seen as generalizations of Abelian arrangements (see Remark \ref {rem:LTA}), we treat Abelian arrangements separately because they (and especially their subclass of toric arrangements) are in the focus of a substantial dedicated literature, see \S \ref{ss:AA1}.

\subsection{Abelian arrangements}\label{ss:AA}
As was briefly discussed in the introduction, one of our main motivations comes from the theory of arrangements, and in particular from the desire to uniformly treat {Abelian arrangements} in a way that generalizes the classical theory of hyperplane arrangements (see, e.g., \cite{DeluMFO}). 
 To make the definition in Section \ref{ss:AA1} slightly more explicit, let $\mathbb G$ stand for one of $\mathbb C, \mathbb C^*$ or $\mathbb E$, an elliptic curve, seen as complex algebraic groups, and let $\Lambda$ be the lattice of group homomorphisms $\mathbb G^d \to \mathbb G$.
Any choice of $a_1,\ldots, a_n\in \Lambda$ and $b_1,\ldots, b_n\in \mathbb G$ determines an arrangement
$$
\mathscr A :=\{H_i:= a_i^{-1} (b_i) \mid i=1,\ldots, n\}
$$
of hypersurfaces in $\mathbb G^d$. We call this an {\em Abelian arrangement}. It is called a {\em linear}, {\em toric}, {\em elliptic} arrangement if $\mathbb G$ is $\mathbb C$, respectively $\mathbb C^*$ or $\mathbb E$. The arrangement is called {\em essential} if the $a_i$'s span a full-rank sublattice of $\Lambda$.

A {\em central} arrangement is one where $b_i=\operatorname{id}_{\mathbb G}$ for all $i=1,\ldots,n$. A deep enumerative-combinatorial study of central arrangements, with special attention to the linear and toric case, has led to the introduction of arithmetic Tutte polynomials \cite{Moc1} and arithmetic matroids \cite{BM,dAM}. Questions about commutative-algebraic interpretations of some of the polynomials arising in this context  led to   attempts at modeling the poset $\mathcal I(\mathscr A)$ of {\em ``independent sets''} in the linear and (central) toric case, defined to be the set of pairs $(X,c)$ where $X$ is a ($\mathbb Q$-)linearly independent subset of $\{a_i\}_i$ and $c$ is a connected component of the intersection of the corresponding hypersurfaces \cite{Lenz,Martino}.\footnote{The definitions in \cite{Lenz,Martino} are formally in terms of pairs $(X,g)$ where $g$ is a torsion element of the quotient group $\Lambda / \langle X \rangle_{\mathbb Z}$. Such torsion elements are however in (natural) bijection with the connected components of the intersection of the hypersurfaces determined by the elements of $X$ (see, e.g., \cite{Moc1,DR}).}
  
In  the general (noncentral) case, one may still look at the poset of {\em layers} $\mathcal C(\mathscr A)$ described in Section \ref{ss:AA1}. The arithmetic matroid of the $\{a_i\}_i$ as well as -- in the linear and toric case -- the rational cohomology algebra of the arrangement's complement can be recovered from $\mathcal C(\mathscr A)$ \cite{dACDMP}. On the other hand, Pagaria \cite{Pagaria2} exhibited a pair of central toric arrangements with isomorphic arithmetic matroids (and matroids over $\mathbb Z $) but non-isomorphic posets of layers.

In order to model these posets we take the approach of \cite{DR}, and consider the (topological) universal covering morphism $$ \mu: \mathbb C^d \to \mathbb G^d.$$ The lift of $\mathscr A$ through this universal covering is a set $\mathscr A^\upharpoonright$ of (affine) complex codimension $1$ subspaces which is invariant under deck transformations. Now, the group of deck transformations acts by translations on $\mathbb C^d$ and is isomorphic to $\mathbb Z^{kd}$, with $k=0,1,2$ according to whether we are in the linear, toric, respectively the elliptic case.

\begin{ex}\label{RE:A} Let us take $\mathbb G=\mathbb C^* $ and $d=3$, so that $\Lambda=\mathbb Z^3$. Consider the arrangement defined by $a_1,\ldots, a_4$ given as the columns of the matrix
$$
\left[\begin{array}{cccc}
1 & 1 & 1 & 3\\
0 & 5 & 0 & 5\\
0 & 0 & 5 & 5
\end{array}
\right]
$$
and $b_1=\ldots = b_4=0$. The associated arrangement $\mathscr A$  is a central and essential arrangement in the torus $(\mathbb C^*)^3$, and it was first considered in \cite{Pagaria}. The arrangement $\mathscr A^{\upharpoonright}$ is the set $\{H_{i,j}\}_{i,j\in \mathbb Z}$ of all hyperplanes $H_{i,j}=\{z\in \mathbb C^3 \mid a_iz^T=j\}$.

\end{ex}

As is well-known \cite{Ard,DR}, every affine hyperplane arrangement such as $\mathcal A^\upharpoonright$ defines a semimatroid whose semilattice of flats is isomorphic to the arrangement's poset of intersections. In our case, associated to $\mathscr A ^\upharpoonright$ we have a semimatroid $\SSS=(S,\mathcal K,\rk)$ with $\mathcal L (\SSS) \simeq \mathcal C (\mathscr A^\upharpoonright)$. 
On this semimatroid the group of deck transformations acts, defining a $\mathbb Z^{kd}$-semimatroid $\mathfrak S_{\mathscr A}$.  

\begin{lem}\label{lem:AbAr} Let $\mathscr A$ be an Abelian arrangement. Then $\mathfrak S_{\mathscr A}$ is well-defined. Moreover,  
\begin{center}
(i) $
\mathcal I_{\mathfrak S_{\mathscr A}}\simeq \mathcal I(\mathscr A)$;
\hspace{2em}
(ii)
$\mathcal{P}_{\mathfrak S _{\mathscr A}}\simeq \mathcal C(\mathscr A)
$;
\hspace{2em}
(iii) $\mathfrak S_{\mathscr A}$ is refined if $\mathscr A$ is essential.
\end{center}
More precisely, $\mathfrak S_{\mathcal A}$ is $0,1,2$-refined according to whether $\mathscr A$ is a linear, toric or elliptic essential arrangement.
\end{lem}

\begin{rem}
If $\mathscr A$ is central, then $T_{\mathfrak S}(x,y)$ corresponds to the arithmetic Tutte polynomial of the list of elements $a_1,\ldots,a_n$ of the Abelian group $\Lambda$, see \cite{Moc1}. 
\end{rem}

\begin{proof}[Proof of Lemma \ref{lem:AbAr}]
We start with a general remark by recalling that $\mathcal K$ is given by all sets of hyperplanes with nonempty intersection. Orbits of $\mathcal K$ under the deck transformation group correspond  bijectively to pairs $(X,c)$ where $X\subseteq \mathscr A$ and $c$ is a connected component of the intersection of the hypersurfaces in $X$.

Since $\mathscr A$ is finite and any intersection has only finitely many components, the finiteness-of-orbits condition in Definition \ref{def_gaos} follows and so $\mathfrak S_{\mathscr A}$ is well-defined.

For (i) and (ii), notice that $\mathcal I (\SSS)$ and $\mathcal L (\SSS)$ are subsets of $\mathcal K$, and orbits of the induced action are, respectively,
\begin{itemize}
    \item[--] for $\mathcal I_{\mathfrak S_{\mathscr A}}$: pairs $(X,c)$ where the characters defining the elements of $X$  are linearly independent (over $\mathbb Q$);
    \item[--] for $\mathcal P_{\mathfrak S_{\mathscr A}}$: pairs $(X,c)$ where the characters defining the elements of $X$  form a subset of $\{a_i\}_{i}$ that is closed under  linear dependency (over $\mathbb Q$).
\end{itemize}
 Comparing these descriptions with the definitions given above, claims (i) and (ii) follow.
 
For (iii) we first notice that the action is translative (see Example \ref{rem:TAC}), then we separate the three cases. In the linear case, the group is trivial, hence the action is clearly $0$-refined.  In the toric and elliptic case we can choose coordinates so that the action of the deck transformation group coincides with addition by elements of the sublattices $L_t:=\mathbb Z^d \subseteq \mathbb C^d$, resp.\ $L_e:=\mathbb Z^d + i\mathbb Z^d \subseteq \mathbb C^d$. The stabilizer subgroup of an affine subspace $W$ equals the stabilizer subgroup of its translate at the origin, hence it is a direct summand of $L_t$, resp.\ $L_e$, of rank equal to the lattice rank of $W\cap L_t$, resp.\ $W\cap L_e$. This rank equals $\dim_{\mathbb C} W$ (resp.\ $2\dim_{\mathbb C} W$). Now, essentiality of $\mathscr A$ implies that the minimal intersections of $\mathscr A^\upharpoonright$ have dimension $0$, and so that the poset rank of $W$ in $\mathcal C (\mathscr A^\upharpoonright)$ equals the codimension of $W$: $\rho(W)=d-\dim_{\mathbb C}W$. The stabilizer of $W$ has then rank $d-\rho(W)$ (resp.\ $2(d-\rho(W))$). Via the isomorphism $\mathcal C(\mathscr A^\upharpoonright ) \simeq \mathcal L (\SSS)$ we conclude that the stabilizer of every subset $X\in \mathcal K$, which coincides with the stabilizer of the intersection associated to $X$, has rank $d-\rho(X)$ (resp.\ $2(d-\rho(X))$). This proves that $\mathfrak S_{\mathscr A}$ is $1$-refined, resp.\ $2$-refined depending on whether we are in the toric or elliptic case.
\end{proof}

We are naturally led to the following definition.

\begin{defi}
Let $\mathscr A$ be an Abelian arrangement. The Stanley-Reisner ring of $\mathscr A$ is $\mathcal R(\mathscr A) := \mathcal R_{\mathfrak S_{\mathscr A}}$.
\end{defi}

 Our point of view allows us to also immediately deduce some properties of those rings for the general case of Abelian arrangements, which we state in the following summary of our results in the general context of the theory of Abelian arrangements.

\begin{thm}\label{thm:AbArr}
Let $\mathscr A$ be an Abelian arrangement (i.e., a linear, toric or elliptic arrangement) with associated $G$-semimatroid $\mathfrak S_{\mathscr A}$.
\begin{itemize}
    \item[(i)] The poset $\mathcal C(\mathscr A)$ is $CM(\delta_{\mathfrak S_{\mathscr A}})$.  Its (topological) Betti numbers are evaluations of the action's Tutte polynomial according to Corollary \ref{cor:HT}.
    \item[(ii)] The simplicial poset $\mathcal I(\mathscr A)$ is $CM(\delta_{\mathfrak S_{\mathscr A}})$.
    \item[(iii)] The arrangement's Stanley-Reisner ring $\mathcal R(\mathscr A)$ is $CM(\delta_{\mathfrak S_{\mathscr A}})$. This ring is isomorphic to the ring of invariants of the Stanley-Reisner ring associated to the periodic hyperplane arrangement $\mathscr A^\upharpoonright$.
    \item[(iv)] The $h$-polynomial of $\mathcal R(\mathscr A)$ is given by the action's Tutte polynomial as in Lemma \ref{lem:pols}.(i).
\end{itemize}
\end{thm}

\begin{proof}
Item (i) is Theorem \ref{thm:PGCM}, item (ii) follows from Theorem \ref{thm_CMSM} and \cite{BiPr,WW}, (iii) combines Theorem \ref{thm:invariant} and Theorem \ref{thm:finitePL}. Finally, (iv) follows with  Lemma \ref{lem:pols}.(i).
\end{proof}

\begin{rem}
Notice that when $\mathscr A$ is central and toric, via the case $k=1$ of Lemma \ref{lem:AbAr} we recover the ring of \cite{Lenz,Martino}, where item (iv) of Theorem \ref{thm:AbArr} is proved in the corresponding situation. 
\end{rem}

\begin{ex}[Continued from Example \ref{RE:A}]\label{RE:B} The additive group of translations $G:=\mathbb Z^3\subseteq \mathbb C^3$ acts on the arrangement $\mathscr A^{\upharpoonright}$ and hence also on its semimatroid. Every intersection $X$ of the hyperplanes from $\mathscr A^{\upharpoonright}$ is parallel to some linear subspace $X_0$ obtained intersecting the hyperplanes of the arrangement $\mathscr A_0:=\{H_{i,0}\}_{i=1,\ldots,5}$, hence has the same stabilizer subgroup. Now, $\stab(X_0)=\mathbb Z^{3}\cap X_0$ is a free abelian group of rank equal to the dimension of $X_0$ and $X$, which -- since $\mathcal L(\mathscr A^\upharpoonright)$ is ordered by reverse inclusion -- equals $(3-\rk(X))$, showing that the action $\mathfrak S: G\circlearrowright \mathcal L(\mathscr A^{\upharpoonright})$ is $1$-refined as is expected for a toric arrangement.

Let us now compute the number $\delta_{\mathfrak S}$. A maximal independent set of the semimatroid associated to $\mathscr A^{\upharpoonright}$ is a triple $B:=\{H_{i_1,j_1}, H_{i_2,j_2}, H_{i_3,j_3}\}$ of hyperplanes whose intersection is nonempty (and has dimension $0$). Now, the stabilizer of $B$ is trivial, therefore $G^{(B)}=G$. We have $H_1^{(B)}=\stab_{G}(B\setminus H_{i_1,j_1})=\mathbb Z^3\cap H_{i_2,0}\cap H_{i_3,0}$, a subgroup of rank $1$ of which we can choose a generator $w_1$, and similarly we can choose generators $w_2$ for $H_2^{(B)}$ and $w_3$ for $H_3^{(B)}$. The $w_1$ are well-defined up to sign reversal. Then, $\delta_{\mathfrak S}(B)=\vert \det(w_1,w_2,w_3) \vert$. The following table lists the different values for all possible $B$.
 
\newcommand{\tvec}[3]{\left(\begin{array}{c}
    #1 \\ #2 \\ #3
    \end{array}\right)}

$$
\begin{array}{c|c|c|c||c}
B \textrm{ (ordered)} & w_1 & w_2 & w_3 & \delta_{\mathfrak S}(B)\\\hline
H_{1,\ast},H_{2,\ast},H_{3,\ast} &
   \tvec{5}{-1}{-1} & \tvec{0}{1}{0} &\tvec{0}{0}{1}& 5 \\
H_{1,\ast},H_{2,\ast},H_{4,\ast} &
   \tvec{5}{-1}{-2} & \tvec{0}{1}{-1} &\tvec{0}{0}{1}& 5 \\
H_{1,\ast},H_{3,\ast},H_{4,\ast} &
   \tvec{5}{-2}{-1} & \tvec{0}{1}{-1} &\tvec{0}{1}{0}& 5 \\
H_{2,\ast},H_{3,\ast},H_{4,\ast} &
   \tvec{5}{-2}{-1} & \tvec{5}{-1}{-2} &\tvec{5}{-1}{-1}& 5 \\         
\end{array}    
$$

We conclude that $\delta_{\mathfrak S} = \operatorname{gcd}\{5,5,5,5\}=5$, hence both $\overline{\mathcal I(\mathscr A)}$ and $\overline{\mathcal C (\mathscr A)}$ are Cohen-Macaulay in characteristic $0$ and every characteristic that does not divide $5$. In fact, as was computed in \cite[Section 8]{PP}, for both posets the first integer homology group is $\mathbb Z / 5\mathbb Z$ (while the second homology group is free in both cases).
\end{ex}

\subsection{$(p,q)$-arrangements} We close by considering a class of arrangements introduced by Liu, Tran and Yoshinaga \cite{LTY}, which we call ``$(p,q)$-arrangements''. Given two integers $p,q\in \mathbb N$, consider the connected Abelian Lie group $$\mathbb G_{p,q}:=(S^1)^p\times \mathbb R^q.$$
Every $a\in \mathbb Z^d$ induces a group homomorphism $\varphi_a:(\mathbb G_{p,q})^d\mapsto \mathbb G_{p,q}, z\mapsto \sum a_iz_i$.
Every finite subset $A\subseteq  \mathbb Z^d\setminus \{0\}$ defines the $(p,q)$-arrangement
$$
\mathscr A:=\{\ker \varphi_{a} \mid a\in A\}.
$$

\begin{rem}\label{rem:LTA}
Central linear arrangements can be naturally regarded as $(0,2)$-arrangements; central toric arrangements as $(1,1)$-arrangements, and central elliptic arrangements as $(2,0)$-arrangements.
\end{rem}

We can identify the universal cover of $\mathbb G_{p,q}$ with $\mathbb R^p\times \mathbb R^q$, and the arrangement $\mathscr A$ in $(\mathbb G_{p,q})^d$ lifts to a $\mathbb Z^{pd}$-periodic arrangement $\mathscr A^{\upharpoonright}$ of affine subspaces in $(\mathbb R^{p}\times\mathbb R^{q})^d$. 

\begin{lem}
The poset of intersections $\mathcal C(\mathscr A^{\upharpoonright})$ is a geometric semilattice. Thus, $\mathscr A$ naturally defines a $\mathbb Z^{pd}$-semimatroid $\mathfrak S_{\mathscr A}$.
\end{lem}

\begin{proof}
Write $(\mathbb G_{p,q})^d$ as $((S^1)^d)^p\times (\mathbb R^d)^q$. Then, for every homomorphism $\varphi_a$ the hypersurface $\ker \varphi_a$ is of the form $(K_a)^{p}\times (H_a^{(0)})^{q}$, where $K_a:=\{x\in (S^1)^d \mid \prod_i (x_i)^{a_i} = 1\}$ and, for every $k\in \mathbb Z$, we define the hyperplane $H_a^{(k)}:=\{y\in \mathbb R^d \mid \sum_i a_iy_i = k\}$. Accordingly, passing to the universal cover, the elements of the arrangement $\mathscr A^{\upharpoonright}$ are all affine subspaces of the form
$$
S_{a,k}:=\underbrace{H_a^{(k)}\times \ldots \times H_a^{(k)}}_{p \textrm{ times}} \times 
\underbrace{H_a^{(0)}\times\ldots\times H_a^{(0)}
}_{q \textrm{ times}}
$$
for a given $k\in\mathbb Z$ and $a\in A$. 

Consider the affine hyperplane arrangement $\mathscr B:=\{H_a^{(k)} \mid a\in A, k\in \mathbb Z\}$ in $\mathbb R^d$ and, for every affine subspace $V$ of $\mathbb R^d$, let $V_0$ denote its translate at the origin. 
Then
$$\mathcal C(\mathscr A^{\upharpoonright})
=\{V^p\times V_0^q \mid V\in \mathcal C(\mathscr B)\}
$$ 
from which we see that there is a poset isomorphism 
$
\mathcal C(\mathscr A^{\upharpoonright}) \simeq \mathcal C(\mathscr B).$
Now, $\mathcal C(\mathscr B)$ is a geometric semilattice because $\mathscr B$ is an arrangement of hyperplanes. This proves that $\mathcal C(\mathscr A^{\upharpoonright})$ is a geometric semilattice.

The action of $\mathbb Z^{pd}$ on $\mathcal C(\mathscr A^{\upharpoonright})$ is induced by its action on $\mathscr A^{\upharpoonright}$ as the deck transformation group of the universal covering of $(\mathbb G_{p,q})^d$, which coincides with the action (by translations) of the discrete subgroup  $\mathbb Z^{pd}\times \{0\}^{qd}\subset \mathbb R^{dp}\times \mathbb R^{dq}$. This action is translative because it falls under Example \ref{rem:TAC}.
 Moreover, the stabilizer of any $V^p\times V_0^q \in \mathcal C (\mathscr A^{\upharpoonright})$ is the subgroup $(\mathbb Z^{pd}\times \{0\}^{qd} )\cap (V_0^p\times V_0^q) $, which is a direct summand of rank $\dim V_0^p = p\dim V$. Now, via the isomorphism $
\mathcal C(\mathscr A^{\upharpoonright}) \simeq \mathcal C(\mathscr B)$ we see that the rank of $V^p\times V_0^q$ in $\mathcal C (\mathscr A^{\upharpoonright})$ equals the rank of $V$ in $\mathcal C (\mathscr B)$, which is the codimension of $V$. We conclude that the action is $p$-refined.
 
\end{proof}

In analogy with the previous sections we make the following definition.

\begin{defi}
Let $\mathscr A$ be a $(p,q)$-arrangement. The Stanley-Reisner ring of $\mathscr A$ is $\mathcal R(\mathscr A) := \mathcal R_{\mathfrak S_{\mathscr A}}$.
\end{defi}

Notice that, if $\mathscr A$ is the $(p,q)$ arrangement associated to a linear, toric or elliptic arrangement (cf.\ Remark \ref{rem:LTA}) we recover the rings defined in Subsection \ref{ss:AA}. In general, we immediately obtain the following analogue of Theorem \ref{thm:AbArr}.

\begin{thm}\label{thm:pqArr}
Let $\mathscr A$ be a $(p,q)$-arrangement with associated $G$-semimatroid $\mathfrak S_{\mathscr A}$.
\begin{itemize}
    \item[(i)] The poset $\mathcal C(\mathscr A)$ is $CM(\delta_{\mathfrak S_{\mathscr A}})$.  Its (topological) Betti numbers are evaluations of the action's Tutte polynomial via the same formula as in  Corollary \ref{cor:HT}.
    \item[(ii)] The simplicial poset $\mathcal I_{\mathfrak S_\mathscr A}$ is $CM(\delta_{\mathfrak S_{\mathscr A}})$.
    \item[(iii)] The arrangement's Stanley-Reisner ring $\mathcal R(\mathscr A)$ is $CM(\delta_{\mathfrak S_{\mathscr A}})$. This ring is isomorphic to the ring of invariants of the Stanley-Reisner ring associated to the periodic subspace arrangement $\mathscr A^\upharpoonright$.
    \item[(iv)] The $h$-polynomial of $\mathcal R(\mathscr A)$ is given by the action's Tutte polynomial as in Lemma \ref{lem:pols}.(i).
\end{itemize}
\end{thm}

\begin{rem}
For a $(p,q)$-arrangement $\mathscr A$, the Tutte polynomial of $\mathfrak S_{\mathscr A}$ in the sense of \cite{DR}, which we mention in the Theorem above, equals the $\mathbb G_{p,q}$-Tutte polynomial considered in \cite{LTY} (via, e.g., Proposition 3.6 and \S 4 of \cite{LTY}). 
\end{rem}

\bibliography{Bib_ISR}{}
\bibliographystyle{plain}

\end{document}